\documentclass[11pt]{article}
\usepackage[table]{xcolor}
\usepackage{graphicx} 
\usepackage{comment}
\usepackage{amsmath,amsfonts,amssymb}
\usepackage{graphicx} \usepackage{multirow} \usepackage{tikz,pgf}
\usepackage{url} \usepackage{amsmath} \usepackage{graphicx}
\usepackage{epsfig} \usepackage{epstopdf} \usepackage{color}
\usepackage{enumitem} \usepackage{ amsfonts} \usepackage{amsmath}
\usepackage{tikz}
\usepackage{pgfmath}
\usepackage{tikz-3dplot}
\usetikzlibrary{math}
\usepackage{amsthm} \parindent 0em \parskip 0.5em
\usepackage{mdframed}
\newmdenv[linecolor=black,skipabove=\topsep,skipbelow=\topsep,
leftmargin=-3pt,rightmargin=-3pt,
innerleftmargin=6pt,innerrightmargin=6pt]{mybox}
 \def\tto{\;{\lower 1pt
\hbox{$\rightarrow$}}\kern -10pt \hbox{\raise 2pt
\hbox{$\rightarrow$}}\;} \def\Hat{\widehat} 
\def\Bar{\overline} \def\ra{\rangle} \def\la{\langle}
\def\ve{\varepsilon} \def\epsilon{\varepsilon} \def\B{\mathbb B}
 \def\R{\mathbb R} 
 \def\N{\mathbb N}

 \def\gph{\mbox{\rm gph}}
 \def\epi{\mbox{\rm epi}}
 
\def\dom{\mbox{\rm dom}}

 \def\emp{\emptyset} 
\def\oR{\Bar{\R}} \def\lm{\lambda}  
   
   \def\Om{\Omega}
 \def\emp{\emptyset} 
\def\oR{\Bar{\R}} \def\lm{\lambda}

\setlist[enumerate,1]{itemsep=0.0ex,parsep=0.5ex,label={\rm(\alph*)},leftmargin=*,
align=left} \newcounter{lk}

\usepackage[colorlinks=true]{hyperref}
\usepackage{comment}
\begin{document} 
\begin{center}
{\sc\bf Qualitative Properties and Generalized Differentiation of Nearly Convex Optimal Value Functions with Applications to Duality in Optimization}\\[1ex]
{\sc V. S. T. Long\footnote{Faculty of Mathematics and Computer
Science, University of Science,
Ho Chi Minh City, Vietnam.}$^,$ \footnote{Vietnam National University, Ho Chi Minh City, Vietnam, (email: vstlong@hcmus.edu.vn).},
{\sc B. S. Mordukhovich}\footnote{Department of Mathematics and Center for Artificial Intelligence and Data Science, Wayne State
University, Detroit, MI 48202, USA (email: aa1086@wayne.edu). Research of this author was partly supported by the US National Science Foundation under grant DMS-2204519 and by the Australian Research Council under Discovery Project DP-190100555.},
{\sc N. M. Nam}\footnote{Fariborz Maseeh Department of Mathematics and Statistics,
Portland State University, Portland, OR 97207, USA (email: mnn3@pdx.edu). }, {\sc L. White}\footnote{Fariborz Maseeh Department of Mathematics and Statistics, Portland State University, Portland, OR
97207, USA (email: whitelen@pdx.edu ).}}
\end{center}
\small{\bf Abstract.} This paper investigates the class of optimal value/marginal functions in the general framework of parametric optimization in locally convex topological vector spaces and their specifications. We address some qualitative and generalized differential properties of such functions under near convexity of their initial data. In this way, several fundamental results on nearly convex sets and functions in infinite-dimensional spaces are established and then applied to the study of optimal value functions. The obtained qualitative properties for this class include lower semicontinuity, Lipschitzian stability,  near convexity, etc. Subsequently, we derive calculus rules and representation formulas for $\epsilon$-subdifferentials of optimal value functions and their Fenchel conjugates in both convex and nonconvex settings. Applications of the main results are given to Fenchel and Lagrangian duality in infinite-dimensional constrained optimization by using Fenchel conjugates of set-valued mappings and the developed tools of near convex analysis. \\[1ex]
{\bf Keywords.} Optimal value functions, Fenchel conjugates, nearly convex sets and functions, constrained optimization, $\epsilon$-subdifferentials, duality in optimization\\[1ex]
\noindent {\bf 2020 Mathematics Subject Classification.} 49J52, 46N10, 90C26, 90C31
	
\newtheorem{theorem}{Theorem}[section]
\newtheorem{proposition}[theorem]{Proposition}
\newtheorem{lemma}[theorem]{Lemma}
\newtheorem{corollary}[theorem]{Corollary}
\theoremstyle{definition}
\newtheorem{remark}[theorem]{Remark}
\newtheorem{definition}[theorem]{Definition}
\newtheorem{example}[theorem]{Example}\vspace*{-0.2in}
	
\normalsize
\section{Introduction}\label{intro}\vspace*{-0.1in}

The optimal value function, often referred to as the marginal function, arises naturally in the study of parametric optimization problems. It quantifies the dependence of the optimal value on perturbations of the data and serves as a fundamental object in sensitivity aspects, variational analysis, and duality theory. Besides these disciplines, the importance of marginal functions has been well recognized in applications to systems control, bilevel programming, statistics, economics, machine learning, and other areas of mathematical and applied sciences; see, e.g., the books \cite{bauschke2017convex,bonnans2000perturbation,mordukhovich2006variational,mordukhovich2018variational,mordukhovich2022convex,Rockvariational,thibault2023,zalinescu2002convex} and the references therein.

The class of the \emph{optimal value/marginal functions} under consideration is defined by
\begin{equation}\label{optimalfunction}
\mu(x): =\inf\big\{\phi(x, y) \mid y \in F(x)\big\}, \; x \in X,
\end{equation}
where $X$ and $Y$ are locally convex Hausdorff topological vector spaces (LCTV spaces for brevity), where $\phi \colon X \times Y \to \overline{\mathbb{R}}: = \mathbb{R} \cup \{-\infty\}\cup\{\infty\}$ is an extended-real-valued function, and where $F \colon X \rightrightarrows Y$ is a set-valued mapping/multifunction. 

Since marginal functions of type \eqref{optimalfunction} are {\em intrinsically nondifferentiable}, even for the case of very simple smooth initial data, {\em generalized differentiation} of them---in one or another sense--has played a crucial role in their study and a variety of applications; see, e.g.,  \cite{an2020differential,bonnans2000perturbation,hai25,mordukhovich2006variational,mordukhovich2005variational,mordukhovich2022convex,mordukhovich2012variational,mordukhovich2009subgradients,Rockvariational,thibault2023,zalinescu2002convex} with the vast bibliographies therein.

The thrust of this paper is to incorporate the notion of {\em near convexity} (in particular, its {\em int-nearly convex} version, which has never been on the front line of generalized convexity) for the study and applications of optimal value functions in LCTV spaces. Near convexity (which sometimes referred to as ``almost convexity") goes back to Minty \cite{minty} in finite-dimensional spaces. Rockafellar \cite{rock70} introduced another notion, called virtual convexity, in the setting of smoothly reflexive Banach spaces (which includes all separable reflexive ones) and showed that his notion reduced to that by Minty in finite-dimensional spaces. In our infinite-dimensional study below, we follow Minty's line in the general LCTV spaces. 

Besides establishing natural conditions for the near convexity of marginal functions, we investigate other qualitative properties of them lower and upper semicontinuity and Lipschitzian behavior. After that, our attention is paid to deriving formulas to represent the Fenchel conjugates ans $\ve$-subgradients of optimal value functions under the near convexity assumptions on the given data in LCTV spaces. 

Applications of the established results are provided to duality in constrained optimization in both Fenchel and Lagrange forms. To our duality investigations, consider the problem
\begin{equation}\label{cvx_opt}
\text{minimize }\;f(y)\;
\text{ subject to}\;y \in \Omega
\end{equation}
with an extended-real-valued cost. To connect \eqref{cvx_opt} with \eqref{optimalfunction}, suppose that $\phi(0, y) = f(y)$ and $F(0) = \Omega$.  Then problem~\eqref{cvx_opt} can be rewritten in the perturbed form
\begin{equation}\label{primal problem}
\text{minimize }\;\phi(0,y)\;
\text{ subject to }\; y \in F(0).
\end{equation}
Based on the obtained formula for the Fenchel conjugate of the optimal value function \eqref{optimalfunction}, we construct a new dual problem for \eqref{primal problem}. This dual problem is defined using the Fenchel conjugate of the cost function $\phi$ and the notion of Fenchel conjugate for set-valued mappings, which has been recently introduced in \cite{nam2024notion}. The established Fenchel duality results are then applied to Lagrangian duality in optimization problems with inequality constraints.

The rest of the  paper is structured as follows. Section~\ref{sectionpre} contains basic definitions and fundamental elements of variational and convex analysis that are used throughout the paper. Section~\ref{near convexity} provides essential tools for analyzing nearly convex sets and functions in infinite-dimensional spaces with the particular emphasis on their properties and behavior under linear operations. In Section~\ref{sectionproperties}, we study qualitative properties of optimal value functions including semicontinuity, near convexity, and Lipschitzian behavior. Section~\ref{sectionsubdifferential} is devoted to calculating Fenchel conjugates of optimal value functions, which Section~\ref{sec:subdiff} establishes precise formulas for representing $\ve$-subgradients of such functions and their conjugates. In Sections~\ref{sec:duality} and \ref{sec:lagr}, we present applications of the obtained results to duality theory of constrained optimization in the Fenchel and Lagrangian frameworks, respectively. Finally, Section~\ref{secconcluion} contains concluding remarks and discusses some directions of the future research.

Unless otherwise stated, the spaces under consideration are {\em LCTV}. The bilinear form $\langle \cdot, \cdot \rangle \colon X^* \times X \to \mathbb{R}$ on the product of $X$ and its dual $X^*$ is defined by $\langle x^*, x \rangle: = x^*(x)$.
\vspace*{-0.2in}

\section{Basic Definitions and Preliminaries}\label{sectionpre}\vspace*{-0.1in}

This section provides a brief overview of fundamental notions and key properties of convex and nearly convex sets, as well as convex and nearly convex functions, which will be used throughout the paper. For a more comprehensive treatment of the corresponding notions of convex analysis, we refer the reader to \cite{mordukhovich2022convex, rockafellar1997convex, zalinescu2002convex}. The reader can also find below more references and discussions on nearly convex sets and functions.

As usual, we adopt the standard conventions of infinite arithmetic, which can be found in \cite{zalinescu2002convex}. In particular, we have $\sup \emptyset =-\infty$, $\inf \emptyset =\infty$,
$$
\infty+\infty=\infty,\; a + \infty = \infty,\;a + (-\infty) = a -\infty = -\infty\;\text{ for any }\;a \in \R.
$$
where $\mathbb R$ stands for the collection of real numbers. Denote
$\mathbb R_+:=[0,\infty)$.

A set $\Om\subset X$ is  \textit{convex} if $\alpha x+(1-\alpha)y\in \Om$ for all $0\leq \alpha \leq 1$ and all $x,y\in \Om$. The topological interior and closure of $\Om$ are denoted by ${\rm int}\, \Om$ and $\overline{\Om}$, respectively. Given a function $f\colon X\to \overline{\mathbb R}$, the \textit{epigraph}, {\em strict epigraph}, and \textit{domain} of $f$ are defined by
$$
\epi\,f:=\{(x,\lambda)\in X\times \mathbb R\;|\; f(x)\leq \lambda\},\;\epi_s\, f=\{(x, \lambda)\;|\; f(x)<\lambda\},\;\dom\, f=\{x\in X\;|\; f(x)<\infty\},
$$
respectively. A function $f$ is \textit{proper} if $\dom\, f\neq \emptyset$ and $f(x)>-\infty$ for all $x\in X$. We also denote and define the \textit{finiteness domain} of $f$ by
\begin{equation}\label{intrinsicdomain}
\mbox{\rm dom}_{\mathbb R}\,f=\{x\in X\mid f(x)\in \R\}.
\end{equation}

Recall that $f$ is a \textit{convex function} if $\epi\, f$ is a convex set in $X\times \mathbb R$. It is easy to see that $f$ is a convex function if and only if $\epi_s\,f$ is a convex set in $X\times \R$.

A function $f$ is \textit{lower semicontinuous} at $x_0\in X$ if for each $\lambda\in \R$ with $\lambda<f(x_0)$, there exists a neighborhood $V$ of $x_0$ such that
$\lambda<f(x)$ for all $x\in V$.  We also say that $f$ is {\em upper semicontinuous} at $x_0$ if the function  $-f$ is lower semicontinuous at this point.  

The {\em closed hull/lower semicontinuous envelope} of $f$ on $X$, denoted by $\Bar f$, is defined as the function for which $\epi\, \Bar f=\Bar{\epi\, f}$. The version of this at a given {\em point} $x_0\in X$ is given by
\begin{equation}\label{closedf}
\Bar f(x_0):=\liminf_{x\to x_0}f(x)=\sup_{V\in \mathcal{N}(x_0)}\inf_{x\in V}f(x)\; \text{ for }x_0\in X,
\end{equation}
where $\mathcal{N}(x_0)$ denotes the collection of all neighborhoods of $x_0$. We say that
$f$ is {\em bounded from below} on a set $\Omega\subset X$ if there exists $\lambda\in \R$ such that
\begin{equation*}
\lambda\leq f(x)\; \mbox{ for all }\;x\in \Omega.
\end{equation*}

The next notion plays a crucial role in duality theory. The \textit{Fenchel conjugate} of $f$ is the function $f^*\colon X^*\to \overline{\mathbb R}$ defined by
$$
f^*(x^*):=\sup\big\{\la x^*,x\ra-f(x)\;\big|\; x\in X\big\}.
$$
The \textit{Fenchel biconjugate}  $f^{**}: X\to\overline{\mathbb R}$ of $f$ is 
given by
$$
f^{**}(x)=\sup\big\{\langle x^*,x\ra -f^*(x^*)\;\big|\; x^* \in X^*\big\}.
$$

Let us now define the central notions used throughout this paper as discussed in Section~\ref{intro}.

\begin{definition}\label{def34}
Given  a set $\Omega\subset X$ and a function  $f\colon X\to \oR$, we say that:

{\bf(i)} The set $\Om$ is {\em nearly convex} if there exists a convex set $C\subset X$ such that
$C\subset \Omega\subset \overline{C}$.

{\bf(ii)} The set $\Omega$ is {\em int-nearly convex} if there exists a convex set $C\subset X$ with $\mbox{\rm int}\,C\neq\emptyset$ such that $C\subset \Omega\subset \overline{C}$.

{\bf(iii)} The function $f$ is {\em nearly convex} (resp., {\em int-nearly convex}) if the set $\epi\, f$ is nearly convex (resp., int-nearly convex) in $X\times\R$.
\end{definition}\vspace*{-0.1in}

Considering further a \textit{set-valued mapping} $F\colon X\rightrightarrows Y$ with its {\em domain} and \textit{graph}
$$
\dom\, F:=\big\{x\in X\;\big|\; F(x)\neq \emptyset\big\},\quad\gph\,F=\big\{(x,y)\in X\times Y\;\big|\; y\in F(x)\big\},
$$
we say that $F$ is \textit{convex} if $\gph\, F$ is a convex set, and that $F$ is {\em nearly convex} (resp., \textit{int-nearly convex}) if $\gph\,F$ is a nearly convex (resp., int-nearly convex) set. The mapping $F$ is said to be \textit{proper} if $\dom\,F\ne\emp$. 

In the definition of nearly convex sets, we follow the pattern suggested by Minty \cite{minty} in finite-dimensional spaces. This pattern has been developed and applied in more recent publications in finite \cite{heinz,bot,huy2023nearly,nam2025near} and infinite \cite{LMN} dimensions. It seems that the particular emphasis on int-near convexity and its applications has never been explored in the literature.

\begin{remark}\label{remark1}
Observe that in Definition~\ref{def34}(i), we have $\overline{\Omega} = \overline{C}$. This tells us that if $f$ is nearly convex in (iii), then its closed hull $\Bar f$ is convex. Consequently, if $f$ is both nearly convex and lower semicontinuous, then it is a convex function.
\end{remark}\vspace*{-0.1in}

 Next we recall the classical definition of $\epsilon$-subgradients for extended-real-valued functions, which may not be necessarily convex or nearly convex.
 
\begin{definition}\label{def:subgrad}
Let $f\colon X\to \overline{\mathbb R}$, let $x_0\in\mbox{\rm dom}_{\mathbb R}\,f$, and let $\epsilon\in\mathbb R_+$. The \textit{$\epsilon$-subdifferential} of $f$ at $x_0$ is the collections of \textit{$\epsilon$-subgradients} defined by
\begin{equation}\label{subgrad}
\partial_\epsilon f(x_0):=\big\{x^*\in X^*\;\big|\: \la x^*, x-x_0\ra \leq f(x)-f(x_0)+\epsilon\;\text{ for all }\;x\in X\big\}.
\end{equation}
We put $\partial_\epsilon f(x_0):=\emptyset$ if  $x_0\notin \mbox{\rm dom}_{\mathbb R}\,f$.
\end{definition}\vspace*{-0.1in}

It is easy to see that $\dom\, \partial_\epsilon f\subset \dom_{\R}\,f$, and that if $\partial_\epsilon f(x_0)\neq \emptyset$ for some $\epsilon\in \mathbb R_+$, then  $f$ is proper. Moreover, for $0\leq \epsilon_1\leq \epsilon_2$ we have
$$
\partial_{0} f(x_0)\subset \partial_{\epsilon_1} f(x_0)\subset \partial_{\epsilon_2} f(x_0)\;\mbox{ and }\;\partial_{\epsilon} f(x_0)=\bigcap_{\eta >\epsilon} \partial_{\eta} f(x_0)\;\text{ whenever }\;\epsilon\in \mathbb R_+.
$$
Given a subset $\Om\subset X$, the \textit{indicator function} $\delta(\cdot; \Om)\colon X\to \overline{\mathbb R}$ and the \textit{support function} $\sigma_\Om: X^*\to \R$ of $\Om$ are defined, respectively, by
$$\delta(x;\Om):=\begin{cases}
0     & \text{ if } x\in \Om, \\
\infty     & \text{ otherwise,}
\end{cases}$$
\begin{equation}\label{supportfunction}
\sigma_\Om(x^*):=\sup\{\la x^*,x\ra\mid x\in \Om\}\;\text{ for }\;x^*\in X^*.
\end{equation}

Given $\epsilon\in\mathbb R_+$, the $\epsilon$-{\em normal cone} of $\Om$ at $x_0\in \Om$ is
\begin{equation}\label{normalcone}
N_\epsilon(x_0;\Om):=\partial_\epsilon \delta(x_0;\Om)=\{x^*\in X^*\;|\; \la x^*, x-x_0\ra\leq \epsilon \text{ for all } x\in \Om\}.
\end{equation}
Recall also that the \textit{infimal convolution} of two proper functions $f,h \colon X\to \overline{\mathbb R}$ is
\begin{equation}\label{inf-conv}
(f\square h)(x):=\inf \{ f(x_1)+h(x_2)\;|\; x_1+x_2=x\}\; \text{for }x\in X.
\end{equation}

We conclude this section by recalling the notion of $\epsilon$-coderivatives for set-valued mappings.

\begin{definition}\label{cod} Given a multifunction $F\colon X\rightrightarrows Y$ and a number $\epsilon\in \mathbb R_+$, the $\epsilon$-\textit{coderivative} of $F$ at $(x_0,y_0)\in\gph\, F$ is a set-valued mapping $D^*_\epsilon F(x_0,y_0)\colon Y^* \rightrightarrows X^*$ defined by
\begin{equation}\label{coderivative}
D^*_\epsilon F(x_0,y_0)(y^*):=\big\{x^*\in X^*\;\big|\; (x^*,-y^*)\in N_\epsilon((x_0,y_0);\gph\, F)\big\},\quad y^*\in Y^*.    
\end{equation}    
\end{definition}\vspace*{-0.1in}

For simplicity, in what follows we skip $\ve=0$ in the subdifferential \eqref{subgrad}, normal cone \eqref{normalcone}, and coderivative \eqref{coderivative}  notations.
\vspace*{-0.2in}

\section{Tools of Near Convexity in LCTV Spaces} \label{near convexity}\vspace*{-0.1in}

This section presents fundamental tools for analyzing {\em int-nearly convex} sets and functions in LCTV spaces. We mainly focus on their properties needed below.

\begin{proposition}\label{NCC} Let $\Omega$ be a subset of $X$. Then $\Omega$ is int-nearly convex if and only if $\overline{\Omega}$ is convex and $\emptyset\neq \mbox{\rm int}\,\overline{\Omega}\subset \Omega$.
\end{proposition}\vspace*{-0.2in}
\begin{proof}
If $\Omega$ is int-nearly convex, then there exists a convex set $C\subset X$ with nonempty interior such that $C\subset \Omega\subset \overline{C}$. Hence $\overline{\Om}=\overline{C}$ and $\overline{\Om}$ is convex. Moreover, it follows that
\begin{equation*}
{\rm int}\, C\subset {\rm int}\, \Omega\subset {\rm int}\,\overline{\Om}= {\rm int}\,\overline{C}.
\end{equation*}
Applying  \cite[Theorem~2.13(ii)]{mordukhovich2022convex} gives us the equality 
${\rm int}\, C={\rm int}\, \overline{C}$, which yields
$$
{\rm int}\, \overline{\Om}={\rm int}\, C\subset C\subset \Om.
$$
Conversely, assuming that $\overline{\Omega}$ is convex and that $\emptyset\neq \mbox{\rm int}\,\overline{\Omega}\subset \Omega$ tells us that the set $C:={\rm int}\,\overline{\Om}$ is convex with nonempty interior satisfying $C\subset\Om$. Since the closure $\overline{\Om}$ is also convex, it follows from  \cite[Theorem~2.13(i)]{mordukhovich2022convex} that
$$
\overline{C}=\overline{{\rm int}\,\overline{\Om}}=\overline{\Om}\supset \Om,
$$
which verifies that $\Om$ is int-nearly convex and thus completes the proof.
\end{proof}

To proceed further, we begin with the following  simple lemma, which follows directly from the proof of Proposition \ref{NCC}.

\begin{lemma}\label{lemmaint}
Let $\Om$ be an int-nearly convex set with $C\subset \Om\subset\overline{C}$, where $C$ is a convex set with nonempty interior. Then we have the equalities
\begin{equation}\label{equ1}
{\rm int}\,C={\rm int}\,\Om={\rm int}\,\overline{\Om}={\rm int}\,\overline{C}.
\end{equation}
\end{lemma}

The next proposition shows that the intersection of int-nearly convex sets is int-nearly convex under a suitable interior condition.

\begin{proposition}\label{lemmaintersectionnearlyconvex}
Let $ \Omega_1, \Omega_2 \subset X $ be int-nearly convex sets. If $({\rm int} \,\Omega_1) \cap ({\rm int} \,\Omega_2) \neq \emptyset$, then the intersection $\Omega_1 \cap \Omega_2$ is int-nearly convex. Furthermore,
\begin{equation}\label{intrule}
    \mbox{\rm int}(\Omega_1\cap \Omega_2)=(\mbox{\rm int}\,\Omega_1)\cap (\mbox{\rm int}\,\Omega_2).
\end{equation}
\end{proposition}\vspace*{-0.2in}
\begin{proof}
Since $\Omega_1$ and $\Omega_2$ are int-nearly convex, there exist convex sets $C_1$ and $C_2$ with ${\rm int}\, C_1 \neq \emptyset$ and ${\rm int}\, C_2 \neq \emptyset$ such that
\[
C_1 \subset \Omega_1 \subset \overline{C_1} \; \text{and} \; C_2 \subset \Omega_2 \subset \overline{C_2}.
\]
Setting $C: = ({\rm int}\, C_1) \cap ({\rm int}\, C_2)$, we get \( C \subset \Omega_1 \cap \Omega_2 \). It follows from Lemma~\ref{lemmaint} that
\[{\rm int}\, C_1 = {\rm int}\, \Omega_1 = {\rm int}\, \overline{C_1} \; \text{ and } \; {\rm int}\, C_2 = {\rm int}\, \Omega_2 = {\rm int}\, \overline{C_2}.\]
Since $({\rm int}\, \Omega_1) \cap ({\rm int}\, \Omega_2) \neq \emptyset$,  we see that \( C \) is a nonempty open convex set. Then
$$
\overline{C} = \overline{({\rm int}\, C_1) \cap ({\rm int}\, C_2)} = \overline{{\rm int}\,C_1} \cap \overline{{\rm int}\,C_2}   = \overline{C_1} \cap \overline{C_2}\supset \Omega_1\cap \Omega_2
$$
due to \cite[Theorems~2.13(a) and 2.14]{mordukhovich2022convex}. Therefore, the set  \( \Omega_1 \cap \Omega_2 \) is int-nearly convex. The proof of \eqref{intrule} is also straightforward. 
\end{proof}\vspace*{-0.1in}

Now we show that the class of int-nearly convex sets is preserved under applying surjective continuous open linear mappings.

\begin{proposition}\label{lemma11}
Let $ \Omega \subset X $ be an int-nearly convex set, and let $T \colon X \to Y $ be a surjective continuous linear operator. Assume that $T$ is an open mapping $($which holds, in particular, if  both $ X $ and $ Y $ are Fr\'echet spaces$)$.  Then the image set $ T(\Omega) $ is int-nearly convex with
\begin{equation*}
{\rm int}\,T(\Omega)=T({\rm int}\,\Om).
\end{equation*}
\end{proposition}\vspace*{-0.2in}
\begin{proof} Since $\Omega$ is int-nearly convex, there exists a convex set $C\subset X$ with ${\rm int}\,C\neq \emptyset$ such that
$C\subset \Omega\subset \overline{C}$. These inclusions imply that
 \begin{equation}\label{near11}
T(C)\subset T(\Om)\subset T(\overline{C})\subset \overline{T(C)},
\end{equation}
where $T(\overline{C}) \subset \overline{T(C)}$ follows from the continuity of $T$. It remains to check that ${\rm int}\,T(C)\neq \emptyset$. Since $T$ is an open mapping, the set $T({\rm int}\,C)$ is open in $Y$. Consequently, we get
\begin{equation}\label{near33}
\emptyset\neq T({\rm int}\,C)\subset{\rm int}\,T(C).
\end{equation}
The convexity of $T(C)$ with ${\rm int}\,T(C)\ne\emp$ tells us  that  $T(\Om)$ is int-nearly convex. Observe that the reverse inclusion in \eqref{near33} always holds true (prove this!). Then Lemma~\ref{lemmaint} yields
\begin{equation*}
T(\mbox{\rm int}\,\Omega)=T(\mbox{\rm int}\,C)=\mbox{\rm int}\, T(C)=\mbox{\rm int}\,\overline{T(C)}=\mbox{\rm int}\,\overline{T(\Omega)}=\mbox{\rm int}\,T(\Omega),
\end{equation*}
which thus completes the proof of the proposition.
\end{proof}\vspace*{-0.1in}

Define the {\em projection mappings} $\mathcal{P}_1\colon X\times Y\to X$ and $\mathcal{P}_{1, 3}\colon X\times Y\times \R\to X\times \R$ by
\begin{equation}\label{proj1}
\mathcal{P}_1(x, y):=x\;\mbox{ and }\;\mathcal{P}_{1, 3}(x, y, \lambda):=(x, \lambda)
\end{equation}
for $(x, y)\in X\times Y$ and  $(x,y,\lm)\in X\times Y\times \R$, respectively.  Given a function $\phi\colon X\times Y\to \oR$ and a set-valued mapping $F\colon X\tto Y$, form the sets
\begin{equation}\label{om12}
\Omega_s:=\epi_s\, \phi,\ \Omega:=\epi\, \phi, \; \mbox{\rm and }\; \Theta:=(\gph\, F)\times \R.
\end{equation}
A useful consequence of Proposition~\ref{lemma11} involving projection mappings is stated below.

\begin{corollary}\label{cor11}
Assume that $\phi$ is an int-nearly convex function and that $F$ is an int-nearly convex 
set-valued mapping such that 
\begin{equation}\label{QC1}
(\mbox{\rm int}\,\Om)\cap (\mbox{\rm int}\,\Theta)\neq\emptyset,\end{equation}
where $\Om$ and $\Theta$ are defined in \eqref{om12}. Then the set $\mathcal{P}_{1,3}(\Om\cap \Theta)$ is int-nearly convex with
\begin{equation*}
{\rm int}\,\mathcal{P}_{1,3}(\Omega\cap \Theta)=\mathcal{P}_{1,3}\big(({\rm int}\,\Om)\cap 
({\rm int}\,\Theta)\big).
\end{equation*}
\end{corollary}\vspace*{-0.1in}
\begin{proof}
By Proposition~\ref{lemmaintersectionnearlyconvex}, the set $\Om\cap \Theta$ is int-nearly convex and \eqref{intrule} holds. Then the conclusion follows directly from Proposition \ref{lemma11} and the fact that $\mathcal{P}_{1,3}$ is a surjective  continuous open linear mapping.
\end{proof}\vspace*{-0.1in}

We conclude this section by presenting several  properties of int-nearly convex functions.

\begin{proposition}\label{pronearlyconvexlip} Let $f\colon X\to \oR$ be a proper int-nearly convex function. Then the following assertions are satisfied:

{\bf(i)} $\Bar{f}$ is a convex function.

{\bf(ii)} $\dom\,f$ is int-nearly convex set such that
\begin{equation}\label{int epi}
\mbox{\rm int}(\epi\, f)=\big\{(x, \lambda)\;\big|\; x\in \mbox{\rm int}(\dom\, f), \; f(x)<\lambda\big\}.
\end{equation}

{\bf(iii)} $\Bar f (x)=f(x)\;\text{ for all } x\in\mbox{\rm int}(\dom\, f).$  

{\bf(iv)} If $X$ is a normed spaces, then $f$ is  locally Lipschitz continuous on ${\rm int}({\rm dom}\,f)$.
\end{proposition}\vspace*{-0.15in}
\begin{proof} (i) It is straightforward to deduce this assertion from Remark~\ref{remark1}. 

(ii) Since $\mathcal{P}_1(\epi\,f)=\dom\,f$, it follows from Proposition~\ref{lemma11} that the set $\dom\,f$ is int-nearly convex. Then we can use \cite[Corollary~6.4]{LMN} to obtain the representation in \eqref{int epi}.

(iii) Lemma~\ref{lemmaint} yields the conditions  
\begin{equation*}\label{gleqmu}
\emptyset\neq {\rm int}(\epi\, f)={\rm int}(\overline{\epi\,f})= {\rm int}(\epi\, \Bar f),
\end{equation*}
and thus it follows from Proposition~\ref{lemma11} that
\begin{equation*}
\mbox{\rm int}(\dom\, f)=\mathcal{P}_1({\rm int}(\epi\,f))=\mathcal{P}_1({\rm int}(\epi\,\Bar f))={\rm int}(\dom\, \Bar f).
\end{equation*}
Pick any $x\in \mbox{\rm int}(\dom\,f)$. Since we always have $\Bar{f}(x)\leq f(x)$, suppose  on the contrary that $\Bar f (x)<f(x)$. Then there exists $\lambda\in \R$ such that $\Bar f (x)<\lambda<f(x)$. By \cite[Proposition~2.2.5]{zalinescu2002convex}, the function $\Bar f$ is proper, and hence we apply \cite[Corollary~2.145]{mordukhovich2022convex} to get
$$
(x,\lambda)\in {\rm int}(\epi\, \Bar f)= {\rm int}(\epi\, f),
$$
which implies by part (ii) that $f(x)<\lambda$, a contradiction justifying (iii). 

(iv) This part follows directly from part (ii) and \cite[Theorem~2.144]{mordukhovich2022convex}.
\end{proof}\vspace*{-0.3in}

\section{Qualitative Properties of Optimal Value Functions}\label{sectionproperties}\vspace*{-0.1in}

In this section, we study several qualitative properties of the optimal value function \eqref{optimalfunction} including semicontinuity, Lipschitzian behavior, and near convexity that are fundamental in variational analysis and optimization. In some cases, the near convexity assumptions play an essential role. We begin with the following description of the domain of \eqref{optimalfunction}.

\begin{proposition}\label{projectdom} Consider the optimal value function \eqref{optimalfunction} and the first projection mapping in \eqref{proj1}. Then we always have the representation
\begin{equation}\label{dom1}
\dom\, \mu=\mathcal{P}_1(\dom\, \phi\cap \gph\,F).
\end{equation}
Consequently, $\mu$ is proper if and only if $\phi(x, \cdot)$ is bounded from below on $F(x)$ for every $x\in X$ and   $\dom\, \phi\cap \gph\,F\neq\emptyset$.
\end{proposition}\vspace*{-0.2in}
\begin{proof} 
Pick $x\in \dom\,\mu$ and find $y\in F(x)$ with $\phi(x, y)<\infty$. Thus $(x, y)\in \dom\,\phi\cap \gph\, F$, which implies that 
$x\in \mathcal{P}_1(\dom\, \phi\cap \gph\, F)$ and readily justifies the inclusion
$$
\dom\, \mu\subset \mathcal{P}_1(\dom\, \phi\cap \gph\,F).
$$ 
The proofs of the reverse inclusion and the subsequent statement are straightforward.    
\end{proof}\vspace*{-0.1in}

Next we obtain relationships between the strict epigraph and epigraph of the optimal value function~$\mu$ and projections of the sets defined in \eqref{om12}.

\begin{proposition}\label{epirep} Consider the optimal value function \eqref{optimalfunction}, the  second projection mapping in \eqref{proj1}, and the sets $\Omega_s$, $\Om$, and $\Theta$ from \eqref{om12}.
Then the following hold:

{\bf(i)} $\epi_s\,\mu=\mathcal{P}_{1,3}(\Omega_s\cap \Theta)$.

{\bf(ii)} $\epi_s\, \mu\subset\mathcal{P}_{1,3}(\Omega\cap \Theta)\subset \epi\,\mu.$
\end{proposition}\vspace*{-0.2in}
\begin{proof} (i) Taking any $(x, \lambda)\in \epi_s\,\mu$, we have 
\begin{equation*}
 (x, \lambda)\in  X\times \R\; \mbox{\rm and }   \mu(x)<\lambda.
\end{equation*}
By definition \eqref{optimalfunction} of $\mu$, there exist $(x, y)\in \gph\, F$ such that $\phi(x, y)<\lambda$, which implies that $(x, y, \lambda)\in \epi_s\, \phi$. This yields $(x, \lambda)\in \mathcal{P}_{1, 3}(\Omega_s\cap \Theta)$ and hence
\begin{equation*}
\epi_s\, \mu\subset \mathcal{P}_{1, 3}(\Omega_s\cap \Theta).
\end{equation*}
To verify the reverse inclusion, pick any $(x, \lambda)\in \mathcal{P}_{1, 3}(\Omega_s\cap \Theta)$ and find $y\in Y$ such that $(x, y, \lambda)\in \epi_s\, \phi$ and $(x, y)\in \gph\, F$, which gives us $\phi(x, y)<\lambda$. Since $y\in F(x)$, we get $\mu(x)\leq \phi(x, y)<\lambda$ and thus $(x, \lambda)\in \epi_s\, \mu$, which completes the proof of (i). 

(ii) Since $\Om_s\subset \Om$, it follows from part (i) that 
$$
\epi_s\,\mu=\mathcal{P}_{1,3}(\Omega_s\cap \Theta)\subset \mathcal{P}_{1,3}(\Omega\cap \Theta).
$$
To justify the second inclusion in (ii), we can argue similarly to the verification of the reverse inclusion in (i). Thus the proof is complete.
\end{proof}\vspace*{-0.1in}

To proceed further with the study of the convexity and int-near convexity of the optimal value function, we need the following lemma.

\begin{lemma}\label{lem0}
Let $f\colon X\to \oR$ be an extended-real-valued function. Then 
$\overline{\epi_s\,f}=\overline{\epi\,f}$.
\end{lemma}\vspace*{-0.2in}
\begin{proof}
Since $\epi_s\,f\subset \epi\,f$, it immediately follows that $\overline{\epi_s\,f}\subset\overline{\epi\,f}$. To verify the reverse inclusion, pick any $(x, \lambda)\in \epi\,f$ and observe that $(x, \lambda+\frac{1}{k})\to (x, \lambda)$ as $k\to \infty$, where $(x, \lambda+\frac{1}{k})\in \epi_s\,f$ for all $k\in \N$. This implies that $(x, \lambda)\in \overline{\epi_s\,f}$. Therefore, we have $\epi\,f\subset \overline{\epi_s\,f}$ and hence $\overline{\epi\,f}\subset \overline{\epi_s\,f}$, which completes the proof.
\end{proof}\vspace*{-0.1in}

Now we are ready to establish two important structural properties of the optimal value function $\mu$ from \eqref{optimalfunction} depending on the convexity/int-near convexity assumptions imposed on the objective function and the constraint mapping. 

\begin{theorem}\label{thmmuconvex}
We have the following properties for the optimal valued function $\mu$:

{\bf(i)} If $\phi$ is a convex function and $F$ is a convex set-valued mapping, then $\mu$ is convex.

{\bf(ii)} If $\phi$ is an int-nearly convex function and $F$ is an int-nearly convex set-valued mapping satisfying condition \eqref{QC1}, then $\mu$ is int-nearly convex.      
\end{theorem}\vspace*{-0.2in}
\begin{proof}
(i) If $\phi$ is a convex function and $F$ is a convex set-valued mapping, then $\epi_s\,\phi$ and $\gph\, F$ are convex sets. Using Proposition~\ref{epirep}(i), we have that the set $\epi_s\,\mu=\mathcal{P}_{1,3}(\Omega_s\cap \Theta)$ is convex because the intersection $\Omega_s\cap \Theta$ is a convex set and $\mathcal{P}_{1,3}$ is a linear mapping. Therefore, $\mu$ is a convex function.

(ii) Proposition~\ref{epirep}(ii) gives us the inclusions
\begin{equation*}
\overline{\epi_s\, \mu}\subset\overline{\mathcal{P}_{1,3}(\Omega\cap \Theta)}\subset \overline{\epi\,\mu}.
\end{equation*}
It follows from Lemma~\ref{lem0} that $\overline{\epi_s\, \mu}=\overline{\epi\,\mu}$ and hence 
\begin{equation}\label{near18}
\overline{\epi\,\mu}=\overline{\mathcal{P}_{1,3}(\Omega\cap \Theta)}.
\end{equation}
Since both  sets $\Om$ and $\Theta$ are int-nearly convex and condition \eqref{QC1} holds, we deduce from Corollary~\ref{cor11} that $\mathcal{P}_{1,3}(\Omega\cap \Theta)$ is a int-nearly convex set. Therefore, Lemma~\ref{lemmaint} tells us that
\begin{equation}\label{near111}
\mbox{\rm int}\,\overline{\mathcal{P}_{1,3}(\Omega\cap \Theta)}=\mbox{\rm int}\,\mathcal{P}_{1,3}(\Omega\cap \Theta)\neq \emptyset.
\end{equation}
By Remark~\ref{remark1}, the set $\overline{\mathcal{P}_{1,3}(\Omega\cap \Theta)}$ is convex, and so is $\overline{\epi\,\mu}$.
Combining \eqref{near18} with \eqref{near111} and Proposition~\ref{epirep}(ii), we obtain the relationships
\begin{equation*}
\emptyset\neq\mbox{\rm int}(\overline{\epi\,\mu})=\mbox{\rm int}\,\overline{\mathcal{P}_{1,3}(\Omega\cap \Theta)}=\mbox{\rm int}\,\mathcal{P}_{1,3}(\Omega\cap \Theta)\subset \mbox{\rm int}(\epi\,\mu)\subset \mbox{\rm epi}\,\mu,
\end{equation*}  
which allow us to deduce from Proposition~\ref{NCC} that $\mu$ is int-nearly convex. 
\end{proof}\vspace*{-0.1in}

Our next goal is to study the semicontinuity properties of the optimal value function $\mu$ in general LCTV spaces.  We begin by recalling the definitions of lower semicontinuity and upper semicontinuity for set-valued mappings. 

\begin{definition}\label{semi}
Let $F \colon X \rightrightarrows Y$ be a set-valued mapping, and let $x_0 \in \dom\, F$.

{\bf(i)} $F$ is \emph{lower semicontinuous} at $x_0$ if for any open set $V$ in $Y$ satisfying $F(x_0)\cap V~\neq \emptyset$, there exists a neighborhood $U$ of $x_0$ in $X$ such that $F(x)\cap V\neq \emptyset$ whenever $x\in U$.

{\bf(ii)} $F$ is \emph{upper semicontinuous} at $x_0$ for any open set $V\supset F(x_0)$, there exists a neighborhood $U$ of
$x_0$ such that $F(U ) \subset V.$ 
\end{definition}\vspace*{-0.1in}

First we provide sufficient conditions for the upper semicontinuity of $\mu$ from \eqref{optimalfunction}.

\begin{proposition}\label{upper} Let $x_0\in \dom_{\R}\,\mu$, where $\dom_{\R}\,\mu$ is defined by \eqref{intrinsicdomain}. Suppose that the constraint mapping $F$ is lower semicontinuous at $x_0$ and that the cost function $\phi$ is upper semicontinuous at every  point $(x_0,y)$ with $y\in F(x_0)$. Then the optimal value function $\mu$ is upper semicontinuous at $x_0$. 
\end{proposition}\vspace*{-0.2in}

\begin{proof} Take any $\lambda\in \R$ such that $\mu(x_0)<\lambda$. Then there exists $y_0\in F(x_0)$ with $\phi(x_0, y_0)<\lambda$. The upper semicontinuity of $\phi$ gives us neighborhoods $U$ of $x_0$ and $V$ of $y_0$ such that $\phi(x, y)<\lambda$ whenever $(x, y)\in U\times V$. Since $F(x_0)\cap V\neq\emptyset$, under the imposed lower semicontinuity of $F$ there exists a neighborhood $U_1$ of $x_0$ for which
$$
F(x)\cap V\neq\emptyset\; \mbox{\rm whenever }x\in U_1.
$$
Setting $W:=U\cap U_1$ and then picking any $x\in W$ and $y\in F(x)\cap V$, we get
\begin{equation*}
\mu(x)\leq \phi(x, y)<\lambda.
\end{equation*}
which therefore verifies the claimed upper semicontinuity of $\mu$.
\end{proof}\vspace*{-0.1in}

Sufficient conditions guaranteeing the lower semicontinuity of the optimal value function are presented in the proposition below.

\begin{proposition}\label{prolowerseimicontinuity} Given $x_0\in \dom_{\R}\,\mu$, suppose that the set $F(x_0)$ is compact. If the mapping $F$ is upper semicontinuous at $x_0$ and the cost function $\phi$ is lower semicontinuous at $(x_0, y)$ for all $y\in F(x_0)$, then $\mu$ is lower semicontinuous at $x_0$. 
\end{proposition}\vspace*{-0.2in}
\begin{proof}
Pick $\lambda \in \R$ with $\lambda<\mu(x_0)$ and find $\gamma\in \R$ such that $\lambda<\gamma<\mu(x_0)$. Then 
\begin{equation*}
\gamma<\mu(x_0)\leq\phi(x_0,y)\; \text{ for all }\;y\in F(x_0).
\end{equation*}
The lower semicontinuity of $\phi$ at $(x_0, y)$ for any fixed $y\in F(x_0)$ yields the existence of (open) neighborhoods $U_y$ of $x_0$ and $V_y$ of $y$ satisfying
\begin{equation}\label{ct3}
\gamma<\phi(x, z)\; \mbox{\rm whenever }\;(x, z)\in U_y\times V_y.
\end{equation}
Since $F(x_0)$ is compact, the open cover $\{V_y\}_{y\in F(x_0)}$ above contains a finite subcover
$$
F(x_0)\subset \bigcup_{i=1}^p V_{y_i}.
$$
The upper semicontinuity $F$ at $x_0\in \dom\, F$ gives us a neighborhood $U_{x_0}$ of $x_0$ with
\begin{equation}\label{ct4}
F(x) \subset \bigcup_{i=1}^p V_{y_i}\; \text{ for all }\;x \in U_{x_0}.
\end{equation}
Define further the open set
\begin{equation*}
U=\bigcap_{i=1}^p U_{y_i}\cap U_{x_0}    
\end{equation*}
and observe that $U$ is a neighborhood of $x_0$. Fixing $x \in U$ such that $F(x)\neq\emptyset$ and taking any $z \in F(x)$, it follows from \eqref{ct4} that $z\in V_{y_i}$ for some $i\in \{1, \ldots, p\}$. Since $(x, z)\in U_{y_i}\times V_{y_i}$, we deduce from \eqref{ct3} that $\gamma<\phi(x, z)$. Thus, $\gamma\leq \mu(x)$. Note that $\mu(x)=\infty$ if $F(x)=\emptyset$. Consequently,
$$
\lambda<\gamma\leq \mu(x)\; \text{ for all } x\in U,
$$
which verifies the claimed lower semicontinuity of $\mu$.
\end{proof}\vspace*{-0.1in}

Now we study Lipschitzian properties of optimal value functions in the setting of normed spaces $(X, \|\cdot\|_X)$ and $(Y, \|\cdot\|_Y)$. Equip the product space $X \times Y$ with the \emph{sum norm}
$\|(x, y)\|: = \|x\|_X + \|y\|_Y$ for all $(x, y) \in X \times Y$ and use the simplified notation $\|\cdot\|$ for the norms on $X$, $Y$, and $X \times Y$ whenever no confusion arises.

Given a single-valued mapping  $f \colon X \to \overline{\mathbb{R}}$ and a set $\Omega\subset \dom_\R\, f$, recall that $f$ is \emph{Lipschitz continuous} on $\Omega$ if there exists a constant $\ell \geq 0$ such that
\begin{equation}\label{lip-f}
|f(x) - f(u)| \leq \ell \|x - u\| \; \text{for all }\;x, u \in \Omega.
\end{equation}
This mapping is \emph{locally Lipschitz continuous} around $x_0 \in  \dom_\R\, f$ if there exist a neighborhood $U$ of $x_0$ and a constant $\ell \geq 0$ such that \eqref{lip-f} holds with $\Om=U$.

Recall further the two (global and local) Lipschitzian properties of set-valued mappings between normed spaces that are used in what follows. 

\begin{definition}\label{lip-set} Let $F \colon X \rightrightarrows Y$ be a set-valued mapping between normed spaces. Then:

{\bf(i)} $F$ has the $($Hausdorff$)$ \emph{Lipschitz continuous} property if there exists $\ell \geq 0$ with
$$
F(u) \subset F(x) + \ell \|x-u\| \mathbb{B}_Y \;\text{for all }\;x, u \in X,
$$
where $\mathbb{B}_Y$ denotes the closed unit ball in $Y$.

{\bf(ii)} $F$ has the $($Aubin$)$ \emph{Lipschitz-like} property around a given pair $(x_0, y_0)\in\gph F$ if there is a constant $\ell \geq 0$ and neighborhoods $U$ of $x_0$ and $V$ of $y_0$ such that
$$
F(u) \cap V \subset F(x) + \ell \|u - x\| \mathbb{B}_Y \; \text{ for all }\;x, u \in U.
$$
\end{definition}\vspace*{-0.1in}

The following theorem provides conditions ensuring the {\em global} Lipschitz continuity property of the optimal value function on its finiteness domain.

\begin{theorem}\label{lipH} For the optimal value function $\mu$ defined by \eqref{optimalfunction} in normed spaces, suppose that the cost function $\phi\colon X\times Y\to \R$ is Lipschitz continuous on $X\times Y$ and the constraint multifunction $F$ is Lipschitz continuous on $X$. Then $\mu$ is  Lipschitz continuous on the set $\dom_{\R}\,\mu$. 
\end{theorem}\vspace*{-0.2in}
\begin{proof}
Taking any $x, u\in \dom_{\R}\,\mu$, we have $x, u\in \dom\, F$. By the imposed Lipschitz continuity of $\phi$, there exists $\ell\ge 0$ ensuring the inequalities
\begin{equation}\label{lip1}
\mu(x)\leq \phi(x, y)\leq \phi(u, y)+\ell\|x-u\|\;\ \mbox{\rm whenever }\;y\in F(x).
\end{equation}
Since the mapping $F$ is Lipschitz continuous, we find $\ell_1\ge 0$ such that   
\begin{equation*}
F(u)\subset F(x)+\ell_1\|x-u\|\B_Y.
\end{equation*}
Thus any $z\in F(u)$ can be written in the form
$z=y+\ell_1\|x-u\|e$ with some $y\in F(x)$ and $e\in \B_Y$. It readily follows from \eqref{lip1} that
\begin{equation*}
\begin{aligned}
\mu(x)&\leq \phi(u, y)+\ell\|x-u\|\\
 &=\phi(u, z-\ell_1\|x-u\|e)+\ell\|x-u\|\\
&\leq \phi(u, z)+\ell_1\ell \|x-u\|+\ell\|x-u\|,
\end{aligned}
\end{equation*}
where the last inequality uses the Lipschitz continuity of $\phi$.
Since $z\in F(u)$ was chosen arbitrarily, taking the infimum over all such $z$ yields
\begin{equation*}
\mu(x)\leq \mu(u)+L\|x-u\|,
\end{equation*}
where $L:=\ell_1\ell+\ell$. This clearly shows that $\mu$ is globally Lipschitz continuous on $\dom_{\R}\,\mu$.
\end{proof}\vspace*{-0.1in}

To proceed with deriving sufficient conditions for local continuity of the optimal value function, we introduce the following notion of  inner semicompactness  of set-valued mappings  important for furnishing limiting procedures; see Definition~1.63 in \cite{mordukhovich2006variational} and the discussion after that definition.

\begin{definition}\label{semi-comp}
Given a  multifunction $G\colon X \rightrightarrows Y$, we say that $G$ is {\em inner semicompact} at $x_0\in X$ if $x_0\in \dom\, G$ and for every sequence $x_k\to x_0$, there exists a sequence $y_k\in G(x_k)$ for all $k$ such that $\{y_k\}$  has a convergent subsequence to some $y_0\in G(x_0)$.
\end{definition}\vspace*{-0.1in}

To establish sufficient condition for {\em local} Lipschitz continuity of the optimal value function, define the {\em solution map} associated with \eqref{optimalfunction} by
\begin{equation}\label{sol}
S(x):=\big\{y\in F(x)\mid \phi(x,y)=\mu(x)\big\}, \quad x\in X.
\end{equation}

\begin{theorem}\label{MLL}
Let $\mu$ be the optimal value function defined by \eqref{optimalfunction} in normed spaces, where $F$ is of closed graph. Then $\mu$ is locally Lipschitzian around a given point $x_0\in \mbox{\rm dom}_{\mathbb R}\,{\mu}$ if the following conditions are satisfied:

{\bf(i)} The solution map $S$ from \eqref{sol} is inner semicompact  at $x_0$; 

{\bf(ii)} For any $y\in S(x_0)$, the cost function $\phi$ is locally Lipschitzian around $(x_0,y)$,   and the constraint multifunction $F$ is Lipschitz-like around the pair $(x_0,y)$.
\end{theorem}\vspace*{-0.2in}
\begin{proof}  Under the inner semicompactness of $S$ and the local Lipschitz property of $\phi$, we can show that $\mu$ is finite on a neighborhood of $x_0$. Suppose on the contrary that $\mu$ is not locally Lipschitzian around $x_0$. Then we can construct sequences $\{x_k\}\subset \dom_{\R}\,\mu$ and $\{u_k\}\subset \dom_{\R}\,\mu$ converging to $x_0$ such that
\begin{equation*}
k\|x_k-u_k\|<|\mu(x_k)-\mu(u_k)|\; \ \mbox{\rm for all }\;k\in \N.
\end{equation*}
Since $S$ is inner semicompact at $x_0$, by using a subsequential argument, we may assume that there exist sequences $y_k\in S(x_k)$ and $z_k\in S(u_k)$ with $y_k\to y_0$ and $z_k\to z_0$. By the closedness of $\gph\, F$, it follows that $(x_0, y_0)\in \gph\, F$ and $(x_0, z_0)\in \gph \, F$. Under the given assumptions, there exist a constant $\ell\geq 0$ and neighborhoods $U$ of $x_0$, $V_1$ of $y_0$ and $V_2$ of $z_0$ such that
\begin{equation}\label{LLI}
|\phi(x, y)-\phi(u, z)|\leq \ell (\|x-u\|+\|y-z\|)
\end{equation}
whenever $(x, y), (u, z)\in U\times V_1$ or $(x, y), (u, z)\in U\times V_2$. Moreover, we have
\begin{equation}\label{API}
F(u)\cap V_i\subset F(x)+\ell\|x-u\|\, \B_Y\;\mbox{ for all }\;x, u\in U\;\mbox{ and }\;i=1, 2.
\end{equation}

For sufficiently large $k$, consider the two possible cases: either $\mu(x_k)>\mu(u_k)$, or $\mu(x_k)<~\mu(u_k)$. In the first case where $\mu(x_k) > \mu(u_k)$, we have
\begin{equation*}\label{locallip1}
k\|x_k-u_k\|<\mu(x_k)-\mu(u_k)=\mu(x_k)-\phi(u_k, z_k).
\end{equation*}
Since $z_k \in F(u_k)$ and $z_k\to z_0$, it follows from \eqref{API} that there exists $\widetilde{y}_k \in F(x_k)$ with $\|z_k-\widetilde{y}_k\|\leq \ell\|x_k-u_k\|$. Applying \eqref{LLI} gives us
\begin{equation*}
\begin{array}{ll}
k\|x_k-u_k\|&<\mu(x_k)-\phi(u_k, z_k)\\
&\leq \phi(x_k, \widetilde{y}_k)-\phi(u_k, z_k)\\
&\leq \ell(\|x_k-u_k\|+\|\widetilde{y}_k-z_k\|)\\
&\leq (\ell+\ell^2)\|x_k-u_k\|.
\end{array}
\end{equation*}
In the second case where $\mu(x_k) < \mu(u_k)$, we have
\begin{equation*}\label{locallip2}
k\|x_k-u_k\|<\mu(u_k)-\mu(x_k)=\mu(u_k)-\phi(x_k, y_k).
\end{equation*}
Using similar arguments as above yields the existence of $\widetilde{z}_k \in F(u_k)$ such that
$\|y_k-\widetilde{z}_k\|\leq \ell\|x_k-u_k\|$, which leads us to the estimates
\begin{equation*} 
\begin{array}{ll}
k\|x_k-u_k\|&<\mu(u_k)-\phi(x_k, y_k)\\
&\leq \phi(u_k, \widetilde{z}_k)-\phi(x_k, y_k)\\
&\leq \ell(\|u_k-x_k\|+\|\widetilde{z}_k-y_k\|)\\
&\leq (\ell+\ell^2)\|x_k-u_k\|. 
\end{array}
\end{equation*}
The obtained contradiction shows that $\mu$ is locally Lipschitzian around $x_0$.
\end{proof}\vspace*{-0.1in}

To conclude this section, we present a new result on the local Lipschitz continuity of $\mu$ under the int-near convexity. This is a direct consequence of 
Proposition~\ref{pronearlyconvexlip}(iv) and Theorem \ref{thmmuconvex}.
	
\begin{proposition}\label{pronearlyconvexlip1} For the optimal value function  $\mu$ defined by \eqref{optimalfunction} in normed spaces, assumed that $\mu$ is proper, $\phi$ is an int-nearly convex function, and $F$ is an int-nearly convex set-valued mapping. Suppose further that condition \eqref{QC1} is satisfied. Then $\mu$ is locally Lipschitz continuous on the interior of its domain, which is a nonempty convex set. 
\end{proposition}\vspace*{-0.3in}

\section{Fenchel Conjugate of Optimal Value Functions}
\label{sectionsubdifferential}\vspace*{-0.1in}

This section is devoted to deriving a representation formula for the Fenchel conjugate of the optimal value function \eqref{optimalfunction} with
int-nearly convex data in LCTV spaces. First we recall the notion of Fenchel conjugates for set-valued mappings recently introduced in \cite{nam2024notion}.

\begin{definition}\label{fen}
Given a set-valued mapping $F\colon X\tto Y$, the \textit{Fenchel conjugate} of $F$ is the extended-real-valued function $F^*\colon X^*\times Y^* \to \oR$ defined by
\begin{equation*}
F^*(x^*,y^*):=\sup\big\{\la x^*,x\ra+\la y^*,y\ra\mid (x,y)\in \gph\, F\big\}\; \text{ for } \;(x^*,y^*)\in X^*\times Y^*.
\end{equation*}
\end{definition}

It is easy to see from the definition of the support function \eqref{supportfunction} that the Fenchel conjugate of $F$ is the support function of $\gph\, F$, i.e.,
\begin{equation}\label{fen-rep}
F^*(x^*,y^*)=\sigma_{\gph\, F}(x^*, y^*)\text{ for } (x^*,y^*)\in X^*\times Y^*.
\end{equation}
Having \eqref{fen-rep} in mind, the following lemma provides a technical tool for representing the Fenchel conjugate of the optimal value function in the int-near convex setting. 

\begin{lemma}\label{lemsigma}
Let $\Om_1$ and $\Om_2$ be int-nearly convex subsets of $X$. Assume that 
\begin{equation}\label{c4.23a}
({\rm int}\, \Om_1)\cap ({\rm int}\, \Om_2)\neq \emptyset.    
\end{equation}
Then for every $x^*\in \dom\,\sigma_{\Om_1\cap \Om_2}$, there exist $x^*_1,x^*_2\in X^*$ such that $x_1^*+x_2^*=x^*$ and
\begin{equation}\label{conj1}
\sigma_{\Om_1\cap\Om_2}(x^*)=\sigma_{\Om_1}(x_1^*)+\sigma_{\Om_2}(x_2^*).
\end{equation}
\end{lemma}\vspace*{-0.2in}
\begin{proof}
Since $\Omega_1$ and $\Omega_2$ are int-nearly convex, there exist convex sets $C_1$ and $C_2$ with ${\rm int}\, C_1 \neq \emptyset$ and ${\rm int}\, C_2 \neq \emptyset$ such that
\begin{equation}\label{est0}
C_1 \subset \Omega_1 \subset \overline{C_1} \; \text{and} \; C_2 \subset \Omega_2 \subset \overline{C_2}.
\end{equation}
It is well known that $\sigma_{C_i}(x^*)=\sigma_{\overline{C}_i}(x^*)$ for all $x^*\in X^*$, $i=1,2$. Therefore, we obtain 
\begin{equation}\label{estimation0}
\sigma_{C_i}(x^*)=\sigma_{\Om_i}(x^*).
\end{equation}
Moreover, it follows from Lemma~\ref{lemmaint} that
$$
{\rm int}\, C_1 = {\rm int}\, \Omega_1\; \text{ and } \; {\rm int}\, C_2 = {\rm int}\, \Omega_2.$$
This together with \eqref{c4.23a} verifies assumption~(a) of \cite[Theorem~4.23]{mordukhovich2022convex}.
Applying that theorem, there exist $x^*_1$ and $x^*_2$ in $X^*$ such that $x_1^*+x_2^*=~x^*$ and $$\sigma_{C_1\cap C_2}(x^*) = \sigma_{C_1}(x_1^*)+\sigma_{C_2}(x_2^*).$$
Combining this with \eqref{est0} and \eqref{estimation0} yields the lower estimate
\begin{equation}\label{etimation1}
\begin{array}{ll}
\sigma_{\Om_1\cap\Om_2}(x^*)&\geq \sigma_{C_1\cap C_2}(x^*) \\
&=\sigma_{C_1}(x_1^*)+\sigma_{C_2}(x_2^*) \\
&=\sigma_{\Om_1}(x_1^*)+\sigma_{\Om_2}(x_2^*).
\end{array}
\end{equation}
On the other hand, we always have the upper estimate
$$
\sigma_{\Om_1\cap\Om_2}(x^*) \leq \sigma_{\Om_1}(x_1^*)+\sigma_{\Om_2}(x_2^*),
$$
which being combined with \eqref{etimation1} verifies the claimed equality \eqref{conj1}.
\end{proof}\vspace*{-0.1in}

Now we establish the main result of this section providing an exact representation of the Fenchel conjugate of optimal value function \eqref{optimalfunction} in the int-nearly convex setting via the infimal convolution \eqref{inf-conv} of the conjugates for $\phi$ and $F$. 

\begin{theorem}\label{thm42} For the optimal value function $\mu$ in \eqref{optimalfunction}, assume that $\phi$ is an int-nearly convex function and $F$ is an int-nearly convex set-valued mapping. Assume further that the qualification condition \eqref{QC1} is fulfilled.
Then for every $x^*\in X^*$, we have the representation 
\begin{equation*}
\mu^*(x^*)=(\phi^*\square F^*)(x^*,0).
\end{equation*}
\end{theorem}\vspace*{-0.2in}
\begin{proof}
Fix $x^*\in X^*$, $y^*\in Y^*$ and pick any $x^*_1, x^*_2\in X^*$ with $x^*_1+x^*_2=x^*$. We get
$$
\begin{array}{ll}
\phi^*(x_1^*,y^*)+ F^*(x_2^*,-y^*)&=\sup\{\la x_1^*,x\ra +\la y^*,y\ra -\phi(x,y)\mid (x,y)\in X\times Y\}\\
&\text{  }+\sup\{\la x_2^*,x\ra -\la y^*,y\ra \mid (x,y)\in \gph\; F\}\\
&\geq \la x_1^*,x\ra +\la y^*,y\ra -\phi(x,y) + \la x_2^*,x\ra -\la y^*,y\ra \\
&= \la x^*,x\ra -\phi(x,y) \text{ whenever }x\in X \text{ and } y\in F(x).
\end{array}$$
Taking the infimum over $y\in F(x)$ tells us that
$$
\phi^*(x_1^*,y^*)+ F^*(x_2^*,-y^*)\geq \la x^*,x\ra - \mu(x)
$$
for every $x\in X$. This leads us to the inequality
$$
\phi^*(x_1^*,y^*)+ F^*(x_2^*,-y^*)\geq \sup\{\la x^*,x\ra - \mu(x)\mid x\in X\}=\mu^*(x^*).
$$
Since $y^*$ and $x^*_1,x^*_2$ were chosen arbitrarily with $x^*_1+x^*_2=x^*$, we have the lower estimate
$$
(\phi^*\square F^*)(x^*,0)\geq \mu^*(x^*).
$$
To verify the upper estimate of the infimal convolution
\begin{equation}\label{conin}
(\phi^*\square F^*)(x^*,0)\leq \mu^*(x^*),
\end{equation}
consider the two possible cases: either $\mu^*(x^*)=\infty$ or $\mu^*(x^*)<\infty$. If $\mu^*(x^*)=\infty$, the estimate in \eqref{conin} holds trivially. Otherwise, we first deduce from \eqref{QC1} that $\mu^*(x^*)>~-\infty$, and hence $\mu^*(x^*)\in \R$. Then it follows from the definitions that
$$
\begin{array}{ll}
\mu^*(x^*)&=\sup\limits_{x\in X}\{\la x^*,x\ra -\mu (x)\}  \\
& =\sup\limits_{x\in X}\Big\{\la x^*,x\ra -\inf\limits_{y\in F(x)} \phi(x,y)\Big\}\\
&=\sup\limits_{x\in X,\; y\in F(x)} \{\la x^*,x\ra -\phi(x,y)\}\\
&= \sup\limits_{(x,y,\lambda)\in \Om\cap \Theta}\la (x^*,0,-1),(x,y,\lambda)\ra\\
&= \sigma_{\Om\cap \Theta}(x^*,0,-1), 
\end{array}
$$ 
where $\Om$ and $\Theta$ are defined by \eqref{om12}. Since the qualification condition \eqref{QC1} holds and $\sigma_{\Om\cap \Theta}(x^*,0,-1)\in\R$, we apply Lemma~\ref{lemsigma} to find $x^*_1,x^*_2\in X^*$, $y^*\in Y^*$, and $\lambda_1,\lambda_2\in\R$ such that $x^*=x^*_1+x_2^*$, $\lambda_1+\lambda_2=-1$, and
\begin{equation}\label{thm423}
\mu^*(x^*)=\sigma_{\Om\cap \Theta}(x^*,0,-1)=\sigma_{\Om}(x_1^*,y^*,\lambda_1)+\sigma_{\Theta}(x_2^*,-y^*,\lambda_2).
\end{equation}
If $\lambda_2\neq 0$, then it follows from the structure of $\Theta=(\gph\,F)\times \R$ that $\sigma_{\Theta}(x_2^*,-y^*,\lambda_2)= \infty$, which leads us to a contradiction. Therefore, we must have $\lambda_2=0$, which yields $\lambda_1=-1$.  Combining the latter with \eqref{thm423} tells us that 
$$\begin{array}{ll}
\mu^*(x^*)&=\sigma_{\Om}(x_1^*,y^*,-1)+\sigma_{\gph\; F}(x_2^*,-y^*)  \\
& =\phi^*(x_1^*,y^*)+ F^*(x_2^*, -y^*)\\
&\geq (\phi^*\square F^*)(x^*,0), 
\end{array}$$
which verifies \eqref{conin} and thus completes the proof of the theorem.
\end{proof}
\vspace*{-0.15in}

\begin{remark}\label{remark43}
If $\phi$ is a convex function and $F$ is a convex set-valued mapping, then the qualification condition \eqref{QC1} can be replaced by the following weaker ones:
\begin{equation*} \label{QCS}
\text{either }\text{\rm int} (\epi\, \phi)\cap ((\gph\, F)\times \R)\neq \emptyset,\; \text{ or }\;\epi\,\phi\cap \text{\rm int} ((\gph\; F)\times \R)\neq \emptyset. 
\end{equation*}
Under these conditions, the conclusion of Theorem~\ref{thm42} still holds.
\end{remark}\vspace*{-0.1in}

The next assertion presents a consequence of Theorem~\ref{thm42} for optimal value functions with no constraint set in \eqref{optimalfunction}. 

\begin{corollary} \label{cor1}
Given the optimal value function $\mu$ in \eqref{optimalfunction} with $\phi$ being int-nearly convex and $F(x)=Y$ for all $x\in X$, we have the conjugate relationship
\begin{equation*}
\mu^*(x^*)=\phi^*(x^*,0)\;\text{ whenever }\;x^*\in X^*.
\end{equation*}
\end{corollary}\vspace*{-0.2in}
\begin{proof} 
The int-near convexity of $\phi$ and  $\gph\,F=X\times Y$ yield
${\rm int}(\epi\, \phi)\neq \emptyset$. Therefore, the qualification  condition \eqref{QC1} is satisfied, and we deduce from Theorem~\ref{thm42} that
\begin{equation}\label{ctthm4.2}
\mu^*(x^*)=(\phi^*\square F)^*(x^*,0)=\inf\{\phi^*(x_1^*, -y^*)+F^*(x_2^*,y^*)\mid x^*_1+x^*_2=x^*,\; y^*\in Y^*\}
\end{equation}
for every $x^*\in X^*$. Observe further that
$$
F^*(x^*,y^*)=\sup\{\la x^*, x\ra+\la y^*,y\ra\mid (x,y)\in X\times Y\}=\begin{cases}
0& \mbox{ if } (x^*,y^*)=(0,0),\\
\infty &\mbox{ otherwise.}
\end{cases}$$
Consequently, the infimum in \eqref{ctthm4.2} is only concerned with the case where $x^*_2=0$ and $y^*=0$. Then we have $x^*_1=x^*$ and $-y^*=0$. Substituting the latter into \eqref{ctthm4.2} gives us the claimed equality $\mu^*(x^*)=\phi^*(x^*,0)$.
\end{proof}\vspace*{-0.1in}

\begin{remark}
In the setting of Corollary \ref{cor1}, instead of using Theorem~\ref{thm42} we can show directly that the conclusion holds for any function $\phi$, without assuming that it is int-nearly convex; see \cite[Theorem~2.6.1(i)]{zalinescu2002convex}.  
\end{remark}\vspace*{-0.3in}

\section{Subdifferentiation of Optimal Value Functions}\label{sec:subdiff}\vspace*{-0.1in}

This section addresses generalized differentiation of the optimal value function \eqref{optimalfunction} under int-near convexity assumptions in LCTV spaces. The main attention is paid to deriving calculation formulas for $\ve$-{\em subdifferentials} \eqref{subgrad} of $\mu$ and its Fenchel conjugate $\mu^*$ in terms of the given data. First we present the following crucial lemma giving us a subdifferential sum rule for int-nearly convex functions, which is certainly of its own interest.

\begin{lemma}\label{lemsubgradient}
Let $g_1,g_2\colon X\to \oR$ be proper int-nearly convex functions, let $g:=g_1+g_2$, and let $\ve\ge 0$. Assume that the following qualification condition holds:
 \begin{equation}\label{nonempty1}
{\rm int}(\epi\, g_1)\cap {\rm int}(\epi\,g_2)\neq \emptyset.	
\end{equation}
Then for any $x_0\in \dom_\R\,g$, we have the sum rule
$$
\partial_\epsilon g(x_0)=\bigcup\big\{\partial_{\epsilon_1}g_1(x_0)+\partial_{\epsilon_1}g_2(x_0)\;\big|\; \epsilon_1,\epsilon_2\geq0\;\epsilon_1+\epsilon_2=\epsilon\big\}.
$$
\end{lemma}\vspace*{-0.2in}
\begin{proof}
Observe first that the inclusion
$$
\partial_\epsilon g(x_0)\supset\bigcup\big\{\partial_{\epsilon_1}g_1(x_0)+\partial_{\epsilon_1}g_2(x_0)\; \big |\;  \epsilon_1,\epsilon_2\geq0\;\epsilon_1+\epsilon_2=\epsilon\big\}
$$
follows directly from the definition of $\epsilon$-subdifferentials. To verify the reverse inclusion, define the two int-nearly convex sets by
$$
\Omega_1:=(\epi\,g_1) \times \mathbb{R}\;\text{ and }\; \Omega_2:=\left\{\left(x, \lambda_1, \lambda_2\right) \in X \times\mathbb{R} \times \mathbb{R} \mid g_2(x)\leq\lambda_2 \right\}.
$$
It follows from \eqref{nonempty1} that there exists a point $(\hat x, \hat\lambda) \in X \times \mathbb{R}$ such that
$$(
x_0, \lambda_0) \in {\rm int}(\epi\, g_1) \cap {\rm int}(\epi\, g_2).
$$
This yields the existence of open sets  $U$ in $X$ and $V$ in $\R$ such that $(x_0, \lambda_0)\in U\times V \subset \operatorname{epi} g_1$ and $(x_0, \lambda_0)\in U \times V \subset \operatorname{epi} g_2$. Then we have 
$$
(\hat x, \hat\lambda, \hat\lambda) \in U\times V\times \R\subset \Omega_1\;\text{ and } \;	(\hat x, \hat\lambda, \hat\lambda) \in  U\times\R\times V \subset\Omega_2
$$
and hence arrive at the condition
\begin{equation}\label{nonempty2}
(\operatorname{int} \Omega_1) \cap (\operatorname{int} \Omega_2) \neq \emptyset.
\end{equation}
Picking any $x^*\in \dom\,g^*=\dom\,(g_1+g_2)^*$ gives us the equalities
$$
\begin{array}{ll}
g^*(x^*)=(g_1+g_2)^*(x^*)&=\sup\{\la x^*,x\ra -(g_1+g_2)(x)\mid x\in \dom(g_1+g_2)\}\\
&=\sup\{\la x^*,x\ra -\lambda_1-\lambda_2)\mid (x,\lambda_1,\lambda_2)\in\Om_1\cap\Om_2\}\\
&=\sigma_{\Om_1\cap\Om_2}(x^*,-1,-1).
\end{array}$$
Since $x^*\in \dom\,g^*$ and $\Om_1\cap \Om_2\neq \emptyset$, it follows that $\sigma_{\Om_1\cap\Om_2}(x^*,-1,-1)\in \R$. By \eqref{nonempty2}, we apply 
Lemma~\ref{lemsigma} to find triples
$(x_1^*,\alpha_1,\beta_1),\,(x_2^*,\alpha_2,\beta_2)\in X^*\times \R\times \R $ 
such that
$(x^*,-1,-1)=(x_1^*,\alpha_1,\beta_1)+(x_2^*,\alpha_2,\beta_2)$ and
$$
g^*(x^*)=\sigma_{\Om_1\cap\Om_2}(x^*,-1,-1)=\sigma_{\Om_1}(x_1^*,\alpha_1,\beta_1)+\sigma_{\Om_2}(x_2^*,\alpha_2,\beta_2).
$$
If $\beta_1\neq 0$, then $\sigma_{\Om_1}(x_1^*,-\alpha_1,-\beta_1)=\infty$, which is impossible. Thus $\beta_1=0$, and similarly we have $\alpha_2=0$. It follows therefore that   \begin{equation}\label{fenchelformula}
\begin{array}{ll}
g^*(x^*)&=\sigma_{\Om_1\cap\Om_2}(x^*,-1,-1)\\
&=\sigma_{\Om_1}(x^*,-1,0)+\sigma_{\Om_2}(x^*,0,-1)\\
&=\sigma_{\epi\,g_1}(x^*,-1)+\sigma_{\epi\,g_2}(x^*,-1)\\
&=g_1^*(x_1^*)+g_2^*(x_2^*).
\end{array}
\end{equation}
Fix any $x^*\in \partial_\epsilon g(x_0)$, where $x_0\in \dom_\R\,g$, and get by definitions that
$$
g(x_0)+g^*(x^*)\leq \la x^*,x_0\ra+\epsilon.
$$
It follows from \eqref{fenchelformula} that there exist $x^*_1,x^*_2\in X^*$ such that $x^*=x^*_1+x^*_2$ and
 \begin{equation*}
g_1(x_0)+ g_2(x_0)+g_1^*(x_1^*)+g_2^*(x_2^*)\leq \la x_1^*,x_0\ra+\la x_2^*,x_0\ra+\epsilon.
\end{equation*}
Hence there exist $\epsilon_1,\epsilon_2\geq 0$ with $\epsilon_1+\epsilon_2=\epsilon$ satisfying 
$$
g_1(x_0)+g_1^*(x_1^*)\leq \la x_1^*,x_0\ra+\epsilon_1\;\text{ and }\;g_2(x_0)+g_2^*(x_2^*)\leq \la x_2^*,x_0\ra+\epsilon_2,
$$
which yields $x^*_1\in \partial_{\epsilon_1} g_1(x_0)$, $x^*_2\in \partial_{\epsilon_2} g_2(x_0)$ and thus finishes the proof.
\end{proof}\vspace*{-0.1in}

For the optimal value function~\eqref{optimalfunction}, take any $\eta > 0$ and $x \in \dom_\R\, \mu$ and then consider the {\em approximate solution map} defined by
\begin{equation}\label{Sapp}
S_\eta(x):=\{y\in F(x)\; |\; \phi(x, y)<\mu(x)+\eta\} \;\text{ for } x\in X.
\end{equation}
It follows from the definition of $\mu$ and $S_\eta$ in \eqref{Sapp} that
\begin{equation}\label{emptyS}
S_\eta(x)\neq\emptyset\;\text{ for all }\;\eta>0.
\end{equation}
The following theorem establishes a calculating formula for the $\epsilon$-subdifferential of the optimal value function $\mu$ via the $\epsilon$-coderivative of the constraint mapping $F$ and the $\epsilon$-subdifferential of the cost function $\phi$. Some steps in the proof below are inspired by the proof of \cite[Theorem~2.6.2]{zalinescu2002convex} that derives related while different results in the convex setting.

\begin{theorem}\label{thm1} Taking $\mu$ from \eqref{optimalfunction}, assume that $\phi$ is an int-nearly convex function and that $F$ is an int-nearly convex set-valued mapping.  Assume further that $\mu$ is proper and that  
\begin{equation}\label{QC14}
{\rm int}(\epi\,\phi)\cap {\rm int}((\gph\, F)\times [0,\infty))\neq \emptyset.
\end{equation}
Then for any $\epsilon\in\mathbb R_+$ and any $x_0\in \dom_\R\,\mu$, we have the representation
\begin{equation*}
\partial_\epsilon\mu(x_0)=\bigcap_{\eta>0}\;\bigcap_{y_0\in S_\eta(x_0)}\;\bigcup_{\epsilon_1+\epsilon_2=\epsilon+\eta,\,
\epsilon_1,\epsilon_2\geq 0}\big\{x_1^*+D_{\epsilon_2}^*F(x_0, y_0)(y_1^*)\; \big|\; (x_1^*, y_1^*)\in \partial_{\epsilon_1} \phi(x_0, y_0)\big\}.
\end{equation*}
\end{theorem}\vspace*{-0.2in}
\begin{proof} 
To verify the inclusion ``$\subset$'' in the claimed formula, take any $x^*\in \partial_\epsilon \mu(x_0)$ and have from the $\epsilon$-subdifferential construction that
\begin{equation}\label{key1}
\la x^*, x-x_0\ra\leq \mu(x)-\mu(x_0)+\epsilon\;\text{ for all } x\in X.
\end{equation}
Picking $\eta>0$ and $y_0\in S_\eta(x_0)$, we get $y_0\in F(x_0)$ and $\phi(x_0, y_0)<\mu(x_0)+\eta$. Combining this with \eqref{key1} yields the estimate
\begin{equation*}
\la x^*, x-x_0\ra\leq \phi(x, y)-\phi(x_0, y_0)+\epsilon+\eta\;\  \mbox{ whenever }\;(x, y)\in \gph\;F.
\end{equation*}
Define the function 
$\varphi(x, y):=\phi(x, y)+\delta((x, y); \gph\;F)\; \text{ for } (x,y)\in X\times Y$
and verify that
\begin{equation}\label{key2}
(x^*,0)\in \partial_{\epsilon+\eta}\varphi(x_0, y_0).
\end{equation}
Indeed, considering the two cases $(x, y)\in \gph\;F$ and $(x, y)\notin \gph\;F$, we obtain
\begin{equation*}
\la x^*, x-x_0\ra+\la 0, y-y_0\ra\leq \varphi(x, y)-\varphi(x_0, y_0)+\epsilon+\eta\; \mbox{\rm whenever }\;(x, y)\in X\times Y
\end{equation*}
in both cases, and thus inclusion \eqref{key2} holds. Observe that $\epi\,\delta(\cdot; \gph F)=(\gph\, F)\times [0, \infty)$.   
Since condition \eqref{QC14} is imposed, we can apply Lemma~\ref{lemsubgradient} to the function $\varphi$ and get
$$
\partial_{\epsilon+\eta}\varphi(x_0,y_0)=\bigcup\big\{\partial_{\epsilon_1} \phi(x_0, y_0)+\partial_{\epsilon_2}\delta((x_0,y_0);\gph\, F)\;\big|\; \epsilon_1,\epsilon_2\geq 0,\; \epsilon_1+\epsilon_2=\epsilon+\eta\big\}.
$$
Combining the latter with \eqref{key2} implies that 
there exist $\epsilon_1, \epsilon_2\geq 0$ with $\epsilon_1+\epsilon_2=\epsilon+\eta$ and
$(x^*_1, y^*_1)\in \partial_{\epsilon_1} \phi(x_0, y_0)$, $(x^*_2,y^*_2)\in \partial_{\epsilon_2}\delta((x_0,y_0);\gph F)$ such that
\begin{equation*}
(x^*,0)=(x^*_1,y^*_1)+(x^*_2,y^*_2).
\end{equation*}
Then it follows from \eqref{normalcone} that $(x^*_2,-y^*_1)\in N_{\epsilon_2}((x_0, y_0);\gph\, F)$. By \eqref{coderivative}, we have
\begin{equation*}
x^*=x^*_1+x^*_2\in  x^*_1+D^*_{\epsilon_2}F(x_0, y_0)(y_1^*),
\end{equation*}
which completes the proof of the inclusion ``$\subset$''.

To verify the reverse inclusion, fix any $x^*$ on the set in the right-hand side therein and for any $\eta >0$ and $y_0\in S_\eta(x_0)$, find $\epsilon_1,\epsilon_2\geq 0$ with $\epsilon_1+\epsilon_2=\epsilon+\eta$ such that
\begin{equation*}
x^*\in x_1^*+D_{\epsilon_2}^*F(x_0, y_0)(y_1^*)
\end{equation*}
for some $(x_1^*, y_1^*)\in \partial_{\epsilon_1}\phi(x_0, y_0)$. By definition of the $\epsilon$-subdifferential, we have
\begin{equation} \label{key3}
\la x_1^*, x-x_0\ra+\la y_1^*, y-y_0\ra\leq \phi(x, y)-\phi(x_0, y_0)+\epsilon_1\; \mbox{\rm for all }\;(x, y)\in X\times Y.
\end{equation}
Fix any $x\in X$ and deduce from $(x^*-x_1^*, -y_1^*)\in N_{\epsilon_2}((x_0, y_0); \gph F)$ that if $y\in F(x)$, then 
\begin{equation*}
\la x^*-x_1^*, x-x_0\ra + \la -y_1^*, y-y_0\ra\leq \epsilon_2.
\end{equation*}
This together with \eqref{key3} implies that
\begin{equation*}
\la x^*, x-x_0\ra \leq \la x_1^*, x-x_0\ra+\la y_1^*, y-y_0\ra+\epsilon_2\leq \phi(x, y)-\phi(x_0, y_0)+\epsilon+\eta
\end{equation*}
whenever $y\in F(x)$. Taking the infimum of $\phi (x,y)$ as 
$y$ ranges over $ F(x)$ gives us
\begin{equation*}
\la x^*, x-x_0\ra\leq \mu(x)-\phi(x_0, y_0)+\epsilon+\eta.
\end{equation*}
Since $y_0\in F(x_0)$, we obtain  $\mu(x_0)\leq \phi(x_0, y_0)$ and hence
\begin{equation*}
\la x^*, x-x_0\ra\leq \mu(x)-\mu(x_0)+\epsilon+\eta.
\end{equation*}
Letting $\eta\to 0^+$ leads us to the inequality
\begin{equation*}
\la x^*, x-x_0\ra\leq \mu(x)-\mu(x_0)+\epsilon,
\end{equation*}
which implies that $x^*\in \partial_\epsilon\mu(x_0)$. This completes the proof of the theorem.  
\end{proof}

\begin{remark}{\rm 
If $\phi$ is a convex function and $F$ is a convex set-valued mapping, then condition \eqref{QC14} can be replaced by the following one:
$$
\phi\text{ is continuous at }(x_0, y_0)\,\text{ for some } (x_0, y_0)\in \dom\, \phi\cap \gph\,F.
$$
Under this continuity assumption, the conclusion of Theorem~\ref{thm1} remains valid. In the special case where
$F(x)=Y$ for all $x\in X$, we recover \cite[Theorem~2.6.2(i)]{zalinescu2002convex}.}
\end{remark}

To proceed further, the next lemma is useful.

\begin{lemma}\label{lem1}
Let $\varphi \colon X\times Y\to \oR$ be defined by 
$$
\varphi(x, y)=\phi(x, y)+\delta((x, y); \gph\;F) \; \text{ for }\;(x,y)\in X\times Y.
$$
Then the optimal value function \eqref{optimalfunction} is expressed by
\begin{equation}\label{ctthm21}
\mu(x)=\inf\{\varphi(x, y)\; |\; y\in F(x)\}.
\end{equation}
Moreover, for every $x_0^*\in X^*$ we have the conjugate function representation
\begin{equation}\label{muvarphi}
\mu^*(x_0^*)=\sup\limits_{(x,y)\in X\times Y}\big\{\la (x_0^*,0),(x,y)\ra-\varphi(x, y)\big\}=\varphi^*(x_0^*,0).
\end{equation}
\end{lemma}\vspace*{-0.2in}
\begin{proof}
Since \eqref{ctthm21} holds obviously, it suffices to verify \eqref{muvarphi}. Pick arbitrary elements $x_0^*\in X^*$ and $x\in X$. If $y\notin F(x)$, then $\varphi(x,y)=\infty$, which yields
$$
\la x_0^*,x\ra- \inf\limits_{y\notin F(x)}\varphi(x,y)=-\infty
$$
and implies therefore that
\begin{equation}\label{varphi}
\sup\limits_{x\in X}\Big\{\la x_0^*,x\ra-\inf\limits_{y\notin F(x)}\varphi(x, y)\Big\}=-\infty.
\end{equation}
It follows directly from the definitions that
$$
\mu^*(x_0^*)=\sup\limits_{x\in X}\Big\{\la x_0^*,x\ra-\inf\limits_{y\in F(x)}\varphi(x, y)\Big\}.
$$
Combining the latter with \eqref{varphi} justifies the claimed representation \eqref{muvarphi}.
\end{proof}\vspace*{-0.1in}

The second theorem of this section provides a precise representation for $\epsilon$-subgradients for the Fenchel conjugate of the optimal value function $\mu$. 

\begin{theorem}\label{thm2} In the setting of Theorem~{\rm\ref{thm1}}, suppose that the conjugate function $\mu^*$ of~\eqref{optimalfunction} is proper. Then we have the representation
\begin{equation*}
\begin{aligned}
\partial_\epsilon \mu^*(x_0^*)=\bigcap_{\eta>0}\mbox{\rm cl}\big\{ x\in X\; \big|\;&\exists\,y\in F(x)\; \mbox{\rm such that } \\
&x^*\in \bigcup_{\epsilon_1+\epsilon_2=\epsilon+\eta, \;\epsilon_i\geq 0}
\{x_1^*+D_{\epsilon_2}^*F(x, y)(y_1^*)\;|\; (x_1^*, y_1^*)\in \partial_{\epsilon_1}\phi(x^*, y^*)\big\}
\end{aligned}
\end{equation*}
whenever $\epsilon\in\mathbb R_+$ and $x_0^* \in \dom_\R\, \mu^*$.
\end{theorem}\vspace*{-0.2in}
\begin{proof} Condition \eqref{QC14} clearly yields \eqref{QC1}, and thus  Theorem~\ref{thmmuconvex}(ii) tells us that $\mu$ is int-nearly convex.
Define the function
\begin{equation}\label{varphi1}
\varphi(x, y):=\phi(x, y)+\delta((x, y); \gph\,F) \; \text{ for }\;(x,y)\in X\times Y.
\end{equation} 
To verify the inclusion ``$\subset$'', pick an arbitrary number $\eta>0$ and deduce from from Remark~\ref{remark1} and \cite[Theorem~2.6.1(vii)]{zalinescu2002convex} by the properness of $\mu^*$ that $\Bar \mu$ is proper and convex. Taking any $x_0\in \partial_\epsilon \mu^*(x_0^*)$, we apply \cite[Theorem~2.3.4(i)]{zalinescu2002convex} to obtain the relationships
$$
\begin{array}{ll}
\Bar{\mu}(x_0)+ \mu^*(x_0^*)&=\mu^{**}(x_0)+\mu^*(x_0^*) \\
&=\sup\{\la x^*, x_0\ra -\mu^* (x^*)\mid x^*\in X^*\}+\mu^*(x_0^*) \\
&\leq \la x_0^*,x_0\ra+\epsilon,\\ 
\end{array}
$$
which lead us to the estimate
$$
\Bar{\mu}(x_0)+ \mu^*(x_0^*)<\la x_0^*, x_0\ra +\epsilon+\frac{\eta}{2}.
$$
It follows from \eqref{closedf} that for every neighborhood $V$ of $x_0$, there exists $x_V\in V$ such that
\begin{equation}\label{ctthm22}
\mu(x_V)+\mu^*(x_0^*) < \la x_0^*, x_V\ra +\epsilon+\frac{\eta}{2}.    
\end{equation}

Since $x_V\in \dom\, \mu$, we deduce from \eqref{emptyS} that there exists $y_V\in S_{\eta/2}(x_V)$, and hence
$$
y_V\in F(x_V)\;\text{ and }\;\phi(x_V,y_V)<\mu (x_V)+\frac{\eta}{2}.
$$
Combining this with \eqref{muvarphi}, \eqref{varphi1} and \eqref{ctthm22} gives us
\begin{equation*}\label{bdt1}
\begin{array}{ll}
\phi(x_V,y_V)+\mu^*(x_0^*)&=\varphi(x_V,y_V)+\mu^*(x_0^*)
\\&= \varphi(x_V,y_V)+\sup\limits_{(x,y)\in X\times Y}\Big\{\la (x_0^*,0), (x,y)\ra -\varphi (x,y)\Big\} \\
&< \la x_0^*, x_V\ra +\epsilon+\eta,
\end{array}
\end{equation*}
which implies that $(x_0^*,0)\in \partial_{\epsilon+\eta} \varphi(x_V,y_V)$.
Using the same arguments as in the proof of Theorem~\ref{thm1} tells us that $(x^*,0)\in \partial_{\epsilon+\eta}\varphi(x_V, y_V)$ if and only if there exist $\epsilon_i\geq 0$ with $i=1, 2$ and $\epsilon_1+\epsilon_2=\epsilon+\eta$ such that 
\begin{equation*}
x_0^*\in x_1^*+D^*_{\epsilon_2}(x_V,y_V)(y^*_1)
\end{equation*}
for some $(x_1^*, y_1^*)\in \partial_{\epsilon_1}\phi(x_V, y_V)$. Since $\eta>0$ was chosen arbitrarily, we have
\begin{equation*}
\begin{aligned}
x_0\in \bigcap_{\eta>0}\mbox{\rm cl}\big\{x\in X\; |\;&\exists y\in F(x)\; \mbox{\rm such that } \\
&x^*\in \bigcup_{\epsilon_1+\epsilon_2=\epsilon+\eta, \;\epsilon_i\geq 0}\{x_1^*+D_{\epsilon_2}^*F(x, y)(y_1^*)\mid (x_1^*, y_1^*)\in \partial_{\epsilon_1}\phi(x, y)\},
\end{aligned}
\end{equation*}
which justifies the inclusion ``$\subset$'' in the claimed formula for $\partial_\epsilon \mu^*(x_0^*)$.
    
To verify  the reverse inclusion, let $x_0$ be an arbitrary element in the set on the right-hand side. Since $x_0^*$ is continuous, for any $\eta >0$ there exists a neighborhood $V_0$ of $x_0$ such that
\begin{equation} \label{keyeta}
\la x_0^*, x\ra < \la x_0^*, x_0\ra+\frac{\eta}{2}\text{ for all } x\in V_0.
\end{equation}
On the other hand, for every neighborhood $V$ of $x_0$, we find by arguing similarly to the proof of Theorem~\ref{thm1} a pair $(x_V,y_V)\in  X\times Y$ such that $x_V\in V_0\cap V$, $y_V\in F(x_V)$ and 
\begin{equation*}
 (x_0^*,0)\in \partial_{\epsilon+\eta/2} \varphi(x_V,y_V).
\end{equation*}
Combining the latter with \eqref{ctthm21}, \eqref{muvarphi}, and \eqref{keyeta} yields
$$
\begin{array}{ll}
\mu(x_V)+ \mu^*(x_0^*) &\leq \varphi(x_V,y_V)+\sup\limits_{(x,y)\in X\times Y}\Big(\la (x_0^*,0), (x,y)\ra -\varphi (x,y)\Big) \\
& \leq \la x_0^*, x_V\ra +\epsilon+\frac{\eta}{2}\\
&< \la x_0^*, x_0\ra +\epsilon+\eta,
\end{array}
$$
which implies in turn that
$$
\inf_{x\in V}\mu(x)+ \mu^*(x_0^*) <\la x_0^*, x_0\ra +\epsilon+\eta.
$$
Combining this with \eqref{closedf} produces 
$$
\Bar{\mu}(x_0)+ \mu^*(x_0^*) \leq \la x_0^*, x_0\ra +\epsilon+\eta.
$$
Letting finally $\eta \to 0$ gives us the relationships
$$
\mu^{**}(x_0) + \mu^*(x_0^*)=\Bar{\mu}(x_0)+ \mu^*(x_0^*) \leq \la x_0^*, x_0\ra +\epsilon,
$$
which show that $x_0\in \partial_{\epsilon}\mu^{*}(x_0^*)$. This verifies the inclusion ``$\supset$'' in the claimed formula and thus completes the proof of the theorem.
\end{proof}\vspace*{-0.25in}

\section{Fenchel Duality in Constrained Optimization}\label{sec:duality}\vspace*{-0.1in}

In this section, we provide applications of the major results established above to the class of constrained optimization problems formulated by
\begin{equation}\label{primal}
\left\{
\begin{array}{ll}
\text{minimize} & f(y)=\phi(0,y), \vspace{1ex} \\
\text{subject to} & y \in \Omega=F(0)
\end{array}
\right.
\end{equation}
under certain relax convexity assumptions with the emphasis on near convexity. Denote by $\mathcal{V}_p$  the optimal value of the primal problem \eqref{primal}, i.e.,
\begin{equation*}\label{Vp}
\mathcal{V}_p:=\inf\{\phi(0, y)\; |\; y\in F(0)\}.
\end{equation*}
The first duality problem associated with \eqref{primal} is defined by
\begin{equation}\label{dual problem2}
\mbox{\rm maximize }-\mu^*(x^*)\;\mbox{ subject to }\;x^*\in X^*.
\end{equation}
Based on Theorem~\ref{thm42} representing the conjugate of the optimal value function via the infimal convolution, we define the second form of the dual problem by
\begin{equation}\label{dual problem}
\mbox{\rm maximize }-(\phi^*\square F^*)(x^*,0)\;\mbox{ subject to }\;x^*\in X^*.
\end{equation}
The optimal values of the dual problems \eqref{dual problem2}  and \eqref{dual problem} are denoted by, respectively,
\begin{equation*}\label{Vd}
\mathcal{V}^1_d:=\sup\big\{-\mu(x^*)\;\big|\; x^*\in X^*\big\}\; \mbox{\rm and }\;\mathcal{V}^2_d:=\sup\big\{-(\phi^*\square F^*)(x^*,0)\;\big|\; x^*\in X^*\big\}.
\end{equation*}

We now present several results linking the primal and dual problems to the optimal value function~$\mu$. The first proposition establishes {\em weak duality} and provides conditions for the equality of the two dual optimal values.

\begin{proposition}\label{proweakduality} We always have the relationships
\begin{equation}\label{weak duality}
\mathcal{V}_p=\mu(0),\quad\mathcal{V}^1_d=\mu^{**}(0) \;\mbox{ and }\;\mathcal{V}^2_d\leq\mathcal{V}^1_d\leq \mathcal{V}_p.
\end{equation}
Moreover, if $\phi$ is an int-nearly convex function, $F$ is an int-nearly convex set-valued mapping, and \eqref{QC1} is satisfied, then $\mathcal{V}^1_d=\mathcal{V}^2_d$.
\end{proposition}
\begin{proof}\vspace*{-0.1in}
The first equality in \eqref{weak duality} follows directly from
\begin{equation*}
\mu(0)=\inf\{\phi(0,y)\; |\; y\in F(0)\}=\mathcal{V}_p.
\end{equation*}
From the construction of the Fenchel biconjugate $\mu^{**}$, we have the second equality in \eqref{weak duality}:
\begin{equation*}
\begin{array}{ll}
\mu^{**}(0)&=\sup\{\la 0, x^*\ra-\mu^*(x^*)\; |\; x^*\in X^*\} =\mathcal{V}^1_d.
\end{array}
\end{equation*}
Since $\mu^{**}(0)\leq \mu(0)$, it follows that $\mathcal{V}_d^1\leq \mathcal{V}_p$. Moreover, we distill from the proof of Theorem~\ref{thm42} above the lower estimate
\begin{equation*}
(\phi^*\square F^*)(x^*, 0)\geq \mu^*(x^*)\; \mbox{\rm for all }x^*\in X^*,
\end{equation*}
which yields the fulfillment of the last assertion in \eqref{weak duality}:
\begin{equation*}
\mathcal{V}^2_d=\sup\{-(\phi^*\square F^*)(x^*,0)\;|\; x^*\in X^*\}\leq \sup\{-\mu(x^*)\; |\; x^*\in X^*\}=\mathcal{V}^1_d\leq \mathcal{V}_p.
\end{equation*}
Under the additional int-near convexity assumptions made, the condition  $\mathcal{V}^1_d=\mathcal{V}^2_d$ follows directly from Theorem~\ref{thm42}, and thus the proof is complete.
\end{proof}\vspace*{-0.1in}

The next theorem establishes a characterization of {\em strong duality} between problems \eqref{primal} and \eqref{dual problem2} by using the subdifferential of the optimal value function.

\begin{theorem}\label{resultdualrelation1}
Consider problem \eqref{primal} and its dual problem in form \eqref{dual problem2}. Then $x_0^*\in \partial \mu(0)$ if and only if $x_0^*$ is a solution to problem \eqref{dual problem2} and we have $\mathcal{V}_p = \mathcal{V}^1_d\in \R$.
\end{theorem}\vspace*{-0.2in}
\begin{proof} 
Given $x_0^*\in \partial \mu(0)$ tells us that $0\in \mbox{\rm dom}_{\R}\,\mu$ and
$$
\la x^*_0, x\ra-\mu(x)\leq -\mu(0)\; \mbox{ for all }\;x\in X.
$$
Taking the supremum with respect to $x\in X$ yields $\mu(0)\leq -\mu^*(x^*_0)$. Combining this with Proposition~\ref{proweakduality}, we get the relationships
$$
\mathcal{V}_p=\mu(0)\leq -\mu^*(x^*_0)\leq \mathcal{V}^1_d=\mu^{**}(0)\leq \mu(0)=\mathcal{V}_p.
$$
Therefore, $\mathcal{V}_p=\mathcal{V}^1_d=\mu(0)\in \R$ and $x_0^*$ is a solution to the dual problem \eqref{dual problem2}.

To verify the reverse implication, let $x_0^*$ be a solution to problem \eqref{dual problem2} and  the equality $\mathcal{V}_p = \mathcal{V}^1_d\in \R$ holds.
Then Proposition~\ref{proweakduality} gives us
$$
-\mu^*(x^*_0)= \mathcal{V}^1_d=\mathcal{V}_p=\mu(0).
$$
Thus $0\in \mbox{\rm dom}_{\R}\, \mu$ and the definition of Fenchel conjugate yields
$$
\sup\{\la x^*_0,x\ra-\mu(x)\mid x\in X\} =-\mu(0)
$$
while implying in turn that
\begin{equation*}
\la x^*_0, x\ra-\mu(x)\leq -\mu(0)\;\mbox{ for all }\;x\in X.
\end{equation*}
This means that $x^*_0 \in \partial \mu(0)$, which completes the proof.    
\end{proof}\vspace*{-0.1in}

The next theorem characterizes strong duality within the framework of near convexity.

\begin{theorem}\label{resultdualrelation}
In the setting of \eqref{optimalfunction}, assume that the cost function $\phi$ is int-nearly convex, that the constraint mapping $F$ is int-nearly convex as well, and that the qualification  condition \eqref{QC1} is satisfied. Given $x_0^*\in X^*$, the following assertions are equivalent:

{\bf(i)} $x_0^*\in \partial \mu(0)$.

{\bf(ii)} $x_0^*$ is a solution to problem \eqref{dual problem2}  and
\begin{equation*} 
\mathcal{V}_p=\mathcal{V}^1_d=\mathcal{V}^2_d\in \R.
\end{equation*}

{\bf(iii)} $(x_0^*,0)$ is a solution to problem \eqref{dual problem} and 
$$
\mathcal{V}_p =\mathcal{V}_1^d=\mathcal{V}^2_d\in \R.
$$ 
\end{theorem}\vspace*{-0.2in}
\begin{proof} Equivalence (i)$\Longleftrightarrow$(ii) follows from 
Proposition~\ref{proweakduality} and Theorem~\ref{resultdualrelation1}. 
Using Theorem~\ref{thm42} together with Proposition~\ref{proweakduality} justifies (ii) by
\begin{equation*}
-\mu^*(x^*_0)=-(\phi^*\square F^*)(x^*_0,0)\leq \mathcal{V}^2_d=\mathcal{V}^1_d.
\end{equation*}
Now suppose that (ii) holds and then get
\begin{equation*}
\mathcal{V}_p=\mathcal{V}^1_d=-\mu^*(x^*_0)=-(\phi^*\square F^*)(x^*_0,0)\leq \mathcal{V}^2_d=\mathcal{V}^1_d=\mathcal{V}_p,
\end{equation*}
which clearly verifies (iii). The proof of (iii)$\Longrightarrow$(ii) is straightforward. 
\end{proof}
\vspace*{-0.1in}

The following statement provides direct sufficient conditions for strong duality.

\begin{corollary}\label{SDCC} Consider the primal problem \eqref{primal problem} and its dual problems \eqref{dual problem2} and \eqref{dual problem}. If $\partial \mu(0)\neq \emptyset$, then we have the strong duality $\mathcal{V}_p=\mathcal{V}^1_d\in \R$.
If in addition both $\phi$ and $F$ are int-nearly convex and \eqref{QC1} holds, then
$$
\mathcal{V}_p=\mathcal{V}^1_d=\mathcal{V}^2_d\in \R.
$$  
\end{corollary}\vspace*{-0.1in}

\begin{remark} The above results on strong Fenchel duality for \eqref{primal} closely relate to the nonemptiness of the subdifferential of the optimal value function $\mu$ from \eqref{optimalfunction} at the origin. Note that the subdifferential definition in \eqref{subgrad} for $\ve=0$ is taken in the form of convex analysis, although the function in question is not assumed to be convex. On the the other hand, it follows from \cite[Theorem~2.4.1(iii)]{zalinescu2002convex} that if $\mu$ is proper on an LCTV space $X$ with the convex domain and if $\partial\mu(x)\ne\emptyset$ on {\em entire  domain} of $\mu$, then the representation in \eqref{subgrad} is {\em equivalent} to the convexity of $\mu$. If the space $X$ is Asplund and $\mu$ is lower semicontinuous on $X$, this equivalence holds when $\partial\mu(x)$ is nonempty on a {\em dense subset} of the domain; see the proof in \cite[Theorem~3.56]{mordukhovich2006variational} based on the approximate mean value theorem. In contrast, our duality results above use the nonemptiness of $\partial\mu$ only at {\em one point}.\vspace*{-0.1in}
\end{remark}

The final result of this section provides sufficient conditions for Fenchel strong duality under explicit assumptions on the original problem.

\begin{proposition}\label{str-duality-convex}
Suppose that $0\in\dom_\R\,\mu$ and that the set $F(0)$ is compact. Suppose further that $F$ is convex and upper semicontinuous at $0$ and that $\phi$ is convex lower semicontinuous at $(0, y)$ for all $y\in F(0)$. Then we have the strong duality $\mathcal{V}_d^1=\mathcal{V}_p\in \R$.
\end{proposition}
\begin{proof}\vspace*{-0.1in}
Under the given assumptions, Theorem~\ref{thmmuconvex}(i) ensures the convexity of $\mu$, while Proposition~\ref{prolowerseimicontinuity} guarantees its lower semicontinuity at $0$.
Moreover, since $0\in\dom_\R\mu$, we apply \cite[Theorem 2.3.4(ii)]{zalinescu2002convex} to obtain $\mu(0)=\mu^{**}(0)$. This yields by Proposition~\ref{proweakduality} the claimed strong duality, and thus completes the proof.
\end{proof}\vspace*{-0.3in}

\section{Applications to Lagrangian Duality}\label{sec:lagr}\vspace*{-0.1in}

Here we consider the optimization problem with functional constraints
\begin{equation}\label{Q4}
\left\{\begin{array}{ll}
\mbox{\rm minimize } &f(y)\\
\mbox{\rm subject to} &  g_i(y) \leq 0\;\mbox{ for }\;i=1, \ldots,m 
\end{array}\right.
\end{equation}
with the functions $f\colon Y\to \R$ and $g_i\colon Y\to \R$ defined on an LCTV space. Denote by ${\cal V}_p$ he optimal value in \eqref{Q4}. Setting $X:=\R^m$ and $g:=(g_1, \ldots, g_m)$, define the {\em Lagrangian}
\begin{equation}\label{lagr}
\mathcal{L}(\lambda,y):=f(y)+\sum_{i=1}^m \lambda_i g_i(y)\;\text{ for }\;(\lambda,y)\in X^*\times Y, 
\end{equation}
where $\lambda:=(\lambda_1\ldots,\lambda_m)\in X^*=\R^m$. The {\em dual Lagrangian} associated with \eqref{lagr} is
\begin{equation}\label{dual-lagr}
\Hat{\mathcal{L}}(\lambda):=\inf_{y\in Y}\mathcal{L}(\lambda,y)\; \text{ for }\;\lambda\in X^*.
\end{equation}

Our goal in what follows is to apply the Fenchel duality results obtained in Section~\ref{sec:duality} to the Lagrangian duality setting by us involving the machinery of optimal value functions. To proceed in this way, define the cost function $\phi\colon X\times Y\to \R$ by
\begin{equation}\label{key29}
\phi(x, y):=f(y)\; \text{ for all }\;(x,y)\in X\times Y    
\end{equation}
 and the constraint mapping $F\colon X\tto Y$ by
\begin{equation}\label{Lagrange Constraint}
F(x):=\{y\in Y\; |\; g_i(y)\leq x_i\; \mbox{\rm for all }\;i=1, \ldots, m\},\quad x=(x_1, \ldots,x_m)\in X.
\end{equation}

First we calculate the Fenchel conjugates of the function $\phi$ and the set-valued mapping $F$ defined in \eqref{key29} and \eqref{Lagrange Constraint}, respectively.

\begin{proposition}\label{fen-lagr} Let $\phi$ be taken from  \eqref{key29} and $F$ be taken from \eqref{Lagrange Constraint}. Then for any $y^*\in Y^*$ and $\lambda=(\lambda_1, \ldots, \lambda_m)\in X^*=\R^m$, we have the formulas
\begin{equation}\label{key31}
\phi^*(\lambda,y^*)=\begin{cases}
\sup\big\{\la y^*,y\ra-f(y)\mid y\in Y\big\}=f^*(y^*)&\text{if }\;\lambda=0,\\
\infty &\text{ otherwise},
\end{cases}
\end{equation}
\begin{equation} \label{key32}
F^*(\lambda,y^*)=\begin{cases}
\sup\Big\{\sum\limits_{i=1}^m\lambda_ig_i(y)+ \la y^*, y\ra\; |\; y\in Y\Big \} &\mbox{if }\;\lambda\leq 0,\\
\infty &\mbox{otherwise}.      
\end{cases}
\end{equation}
\end{proposition}\vspace*{-0.1in}
\begin{proof} It follows from \eqref{key29} and the conjugate definitions that
$$
\begin{array}{ll}
\phi^*(\lambda,y^*)&=
\sup\Big\{\sum\limits_{i=1}^m\lambda_i x_i +\la y^*,y\ra-\phi(x,y)\mid (x,y)\in X\times Y\Big\}\\
 &=\sup\Big\{\sum\limits_{i=1}^m\lambda_i x_i +\la y^*,y\ra-f(y)\mid (x,y)\in X\times Y\Big\},
\end{array}
$$
where $x=(x_1\ldots,x_m)\in X=\R^m$. We clearly have that if $(\lambda_1,\ldots,\lambda_m)\neq (0,\ldots,0)$, then $\phi^*(\lambda,y^*)=\infty$. Otherwise, it follows that
$$
\phi^*(\lambda,y^*)=
\sup\{\la y^*,y\ra-f(y)\mid y\in Y\},
$$
and hence \eqref{key31} holds. We also get by definition that
\begin{equation*}
\begin{aligned}
F^*(\lambda,y^*)=\sup\Big\{\sum\limits_{i=1}^m\lambda_i x_i+ \la y^*, y\ra\; |\; g(y)\leq x,\; x\in X,\; y\in Y\Big\}.    
\end{aligned}
\end{equation*}
Consider the case where $\lambda_i>0$ for some $i=1, \ldots, m$ and then see that $F^*(\lambda,y^*)=\infty$. In the opposite case where $\lambda_i\leq 0$ for all $i=1, \ldots, m$), it follows that
$$
\begin{array}{ll}
F^*(\lambda,y^*)&=\sup\{\la \lambda,x\ra+\la y^*, y\ra\; |\; g(y)\leq x,\; x\in X, \, y\in Y\} \\[0.1in]
&=\sup\{\la \lambda,x\ra+\la y^*, y\ra\; |\; g(y)+\alpha= x,\; \alpha\geq 0,\; x\in X, \, y\in Y\} \\[0.1in]
&=\sup\{\la \lambda,g(y)\ra+\la \lambda,\alpha\ra+\la y^*, y\ra\; |\; \alpha\geq 0,\; y\in Y\} \\[0.1in]
&=\sup\{\la \lambda,g(y)\ra+\la y^*, y\ra\; |\; y\in Y\} \\[0.1in]
&= \sup\Big\{\sum\limits_{i=1}^m \lambda_ig_i(y)+\la y^*, y\ra\; |\; y\in Y\Big\},
\end{array}
$$
which readily justifies the fulfillment of \eqref{key32} and thus completes the proof.
\end{proof}\vspace*{-0.1in}

Next we show that the dual problem \eqref{dual problem} with the data from \eqref{key31} and \eqref{key32}  can be reformulated as the Lagrangian dual form by using the dual Lagrangian \eqref{dual-lagr}.

\begin{theorem}\label{LDF} Let the cost function $\phi$ and the constraint mapping $F$ be given by \eqref{key29} and \eqref{Lagrange Constraint}, respectively. Then for every $\lambda=(\lambda_1,\ldots,\lambda_m)\in X^*$, we have
\begin{equation*}\label{duallagrane1}
\mu^*(-\lambda)=\begin{cases}
-\Hat{\mathcal L}(\lambda) & \text{  if } \lambda\geq 0,\\
\infty & \text{ otherwise.}
\end{cases}
\end{equation*}
\end{theorem}\vspace*{-0.2in}
\begin{proof} It follows from \eqref{key29} and \eqref{Lagrange Constraint} that the optimal value function \eqref{optimalfunction} reduces to
\begin{equation*}\label{optc}
\mu(x)=\inf\{f(y)\; |\; g(y)\leq x,\; y\in Y\} \; \text{ for }\;x\in X.
\end{equation*}
Using the corresponding definitions leads us to the equalities     
$$
\begin{array}{ll}
\mu^*(-\lambda)&=\sup\limits_{x\in X}\{-\la \lambda, x\ra-\mu(x)\} \\[0.1in]
&=\sup\limits_{x\in X}\Big\{-\la \lambda, x\ra-\inf\limits_{y\in Y,\; g(y)\leq x}f(y)\Big\} \\[0.1in]
&=\sup\limits_{x\in X}\sup\limits_{y\in Y,\; g(y)\leq x}\{-\la \lambda, x\ra-f(y)\}\\[0.1in]
&= \sup\limits_{y\in Y}\sup\limits_{x\in X,\; g(y)\leq x}\{-\la \lambda, x\ra-f(y)\}.
\end{array}$$
Due to the obvious representation 
$$
\sup\limits_{x\in X,\;g(y)\leq x}\{-\la \lambda, x\ra-f(y)\}=\begin{cases}
-\la \lambda,g(y)\ra -f(y) &\text{ if } \lambda \geq 0,  \\
\infty&\text{ otherwise}, 
\end{cases}$$ 
we arrive at the claimed expression
$$
\mu^*(-\lambda)=\sup\limits_{y\in Y}\sup\limits_{x\in X,\; g(y)\leq x}\{-\la \lambda, x\ra-f(y)\}=\begin{cases}
-\Hat{\mathcal L}(\lambda) &\text{ if } \lambda \geq 0,  \\
\infty&\text{ otherwise} 
\end{cases}
$$
and thus complete the proof of the theorem. 
\end{proof}\vspace*{-0.1in}

Based on Theorem~\ref{LDF}, the dual problem \eqref{dual problem2} for \eqref{Q4} can be written as
\begin{equation*}
\mbox{\rm maximize }\Hat{\mathcal{L}}(\lambda)\;
\mbox{ subject to }\;\lambda\geq 0
\end{equation*}
with the associated value function $\mathcal{V}_d:=\sup_{\lambda\geq 0}\Hat{\mathcal{L}}(\lambda)$. The following result, presented in  \cite[Theorem~4.3.7]{borwein2006convex} in finite dimensions, is a consequence of our more general developments.

\begin{corollary}
Consider problem \eqref{Q4}, where $f$ and $g_i$, $i=1,\ldots,m,$ are convex and continuous. Impose the Slater constraint qualification: there exists $x_0\in X$ such that $g_i(x_0)<0$ for all $i=1, \ldots, m$. Then we have the strong Lagrangian duality $\mathcal{V}_p=\mathcal{V}_d$.
\end{corollary}\vspace*{-0.2in}
\begin{proof} In the case where $\mathcal{V}_p=\mu(0)=-\infty$, it follows from 
Proposition~\ref{proweakduality} and Theorem~\ref{LDF} that $\mathcal{V}_d=-\infty$, and hence the conclusion holds. Consider the case where $\mathcal{V}_p>-\infty$. Since $\mathcal{V}_p=\mu(0)<\infty$ in this case, we get $0\in \mbox{\rm dom}_\R\,\mu$. Denote further the numbers $\gamma_i:=g_i(x_0)<0$ for $i=1, \ldots, m$ and define the open set
\begin{equation*}
V:=\prod_{i=1}^m (\gamma_i, \infty).
\end{equation*}
It is easy to see that $0\in V\subset \mbox{\rm dom}\,\mu$ and hence $0\in \mbox{\rm int}(\dom\, \mu)$, which yields $\partial \mu(0)\neq\emptyset$. Thus the conclusion follows directly from Corollary~\ref{SDCC} since $\mathcal{V}_d=\mathcal{V}^1_d$ in this setting.     
\end{proof}\vspace*{-0.3in}

\section{Concluding Remarks}\label{secconcluion}\vspace*{-0.01in}

This paper has systematically examined several key properties of the optimal value functions for perturbed optimization problems in LCTV spaces. These properties include near convexity, semicontinuity, and Lipschitz behavior in both convex and nonconvex settings with the emphasis on int-near  convexity in the assumptions  and conclusions. We calculated the Fenchel conjugates and $\ve$-subdifferentials of the optimal value functions and present applications to both Fenchel and Lagrangian duality in constrained optimization.

The obtained results enhance theoretical foundations of variational and near convex analysis providing a basis for future research on duality theory, generalized subdifferential and conjugate calculi with subsequent applications to numerical methods of optimization.\vspace*{0.4in}

\textbf{Data Availability} The paper has no associated data.

\section*{Declarations}

\textbf{Competing Interests} There is no competing interest in this paper.

\end{document}